\documentclass{amsart}
\usepackage{graphicx}
\usepackage{graphics}
\usepackage{psfrag}
\usepackage[active]{srcltx}
\usepackage{amssymb}
\usepackage{amscd}
\theoremstyle{plain} %% This is the default
\newtheorem{theorem}{Theorem}
\theoremstyle{definition}
\newtheorem{definition}{Definition}
\theoremstyle{lemma}
\newtheorem{lemma}{Lemma}
\theoremstyle{remark}

\newif\ifShowLabels
\ShowLabelstrue
\newdimen\theight
\def\TeXref#1{%
     \leavevmode\vadjust{\setbox0=\hbox{{\tt
               \quad\quad #1}}%
                    \theight=\ht0
                         \advance\theight by \dp0
                              \advance\theight by \lineskip
                                   \kern -\theight \vbox to
                                             \theight{\rightline{\rlap{\box0}}%
                                                   \vss}%
                                                         }}%

\allowdisplaybreaks[1]
\newcommand{\refL}[1]{Lemma~\ref{L:#1}}

\newcommand{\refT}[1]{Theorem~\ref{T:#1}}

\newcommand\refEq[1]{(\ref{Eq:#1})}
\newcommand\refS[1]{Section~\ref{S:#1}}

\newcommand\refTa[1]{Table~\ref{Ta:#1}}

\mathcode`\;="203B \mathchardef\sim="2218

\makeatletter
\def\partauthor{%
  \def\author##1{\newline\phantom{Part 1.\enspace}{\mdseries\scshape##1}}}
\def\@part[#1]#2{%
  \ifnum \c@secnumdepth >-2\relax \refstepcounter{part}%
    \addcontentsline{toc}{part}{\partname\ \thepart.%
        \protect\partauthor\protect\enspace\protect\noindent#1}%
  \else
    \addcontentsline{toc}{part}{#1}\fi
  \begingroup
  \def\author##1{\\ \hfil \\ \fontsize{\@xivpt}{20}\mdseries\itshape ##1}%
  \centering
  \ifnum \c@secnumdepth >-2\relax
       {\fontsize{\@xviipt}{22}\bfseries
         \partname\ \thepart} \vskip 20\p@ \fi
  \fontsize{\@xxpt}{25}\bfseries
      #1\vfil\vfil\endgroup \newpage\thispagestyle{empty}}
\def\@spart#1{\addcontentsline{toc}{part}%
    {\protect\partauthor\protect\noindent#1}%
  \begingroup
  \def\author##1{\\ \hfil \\ \fontsize{\@xivpt}{20}\mdseries\itshape ##1}%
  \centering
  \fontsize{\@xxpt}{25}\bfseries
     #1\vfil\vfil\endgroup \newpage\thispagestyle{empty}}
\makeatother

\newcommand\pl{+_3}

\newcommand\comz{\f C}

\newcommand\opar{\textnormal{(}}     % open parentheses
\newcommand\cpar{\textnormal{)}}     % close parentheses
\newcommand\qeddef{\qed}

\newcommand\relsum{\setbox0\hbox{$+$}\ensuremath
    \mathop{\rlap{\lower.31ex\hbox to\wd0{\hfil\kern-.02em \upshape,\hfil}}+}}

\newcommand\vp{\varphi}

\newcommand\pair[2]{(#1,#2)}

\newcommand\smbcomma{\,,\,}

\newcommand\conv{{}^{\scriptstyle\smallsmile}}
\newcommand{\mynegspace}{{\mskip-1.4mu}}

\newcommand\ident{1\mynegspace\textnormal{\rq}}
\newcommand\diver{0\textnormal{\rq}}

\newcommand\cc[2]{#1_{#2}}

\newcommand\com[1]{{\mathfrak{Cm}}(#1)}

\newcommand\inv{{}^{\scriptstyle -1}}
\newcommand\ssm{^{\scriptstyle\smallsmile}}

\newcommand\tsum{\textstyle \sum}

\newcommand\myspace{{\hspace*{.2pt}}}
\newcommand\mmyspace{{\hspace*{.5pt}}}

\newcommand\mo{^{-1}}
\newcommand\f[1]{{\mathfrak{#1}}}
\newcommand\ff[2]{{\mathfrak{#1}_{#2}}}

\newcommand\rp{\mmyspace\vert\mmyspace}

\newcommand\id[1]{id_{#1}}
\newcommand\di[1]{di_{#1}}
\newcommand\scir{\raise2pt\hbox{$\,\scriptscriptstyle\circ\,$}}

\newcommand\al{\alpha}
\newcommand\be{\beta}
\newcommand\ga{\gamma}

\newcommand\per{\textnormal{\myspace.\ }}

\newcommand\comma{\textnormal{\myspace,\ }}
\newcommand\semicolon{\textnormal{\myspace;\ }}

\newcommand{\mc}[1]{\mathcal{#1}}

\newcommand\ii{\iota}

\newcommand\comment[1]{}

 \makeatletter
  \def\SpanTwo{\iffalse{\fi
    \def\ampersand@{&\omit\span\omit\span}%
    \iffalse}\fi
  }
  \makeatother

\newcommand\varnot{\varnothing}

\allowdisplaybreaks[1]

\begin{document}\comment{
\title[The independence of Tarski's axioms]{The independence of Tarski's axioms for the theory of relation
algebras}}
\title[Axiomatic foundations of the calculus of relations]{On Tarski's axiomatic foundations of the calculus of relations}
\author{Hajnal Andr\'eka}
\address{ Alfr\'ed R\'enyi Institute of Mathematics\\ Hungarian Academy of Sciences\\ Budapest, PF.\,127\\
H-1364 Hungary.} \email{andreka@renyi.hu}
\author{Steven Givant}
\address{ Mills College\\ Department of Mathematics and Computer Science\\ 5000 MacArthur Boulevard\\ Oakland, CA 94613\\ United States}
\email{givant@mills.edu}
\author{Peter Jipsen}
\address{Chapman University\\ Faculty of Mathematics\\ School of Computational Sciences\\ 545 West Palm Avenue\\
Orange, CA 92866\\ United States}
\email{jipsen@chapman.edu}
\author{Istv\'an N\'emeti}
\address{ Alfr\'ed R\'enyi Institute of Mathematics\\ Hungarian Academy of Sciences\\ Budapest, PF.\,127\\
H-1364 Hungary.} \email{nemeti.istvan@mta.renyi.hu}

%\thanks{}%


 %\keywords{}%

%\date{}%
%\dedicatory{}%
%\commby{}%

\begin{abstract} It is shown that Tarski's set of ten axioms for the
calculus of relations is independent in the sense that no axiom
can be derived from the remaining axioms.  It is also shown that
by modifying one of Tarski's  axioms slightly, and in fact by
replacing the right-hand distributive law for relative
multiplication with its left-hand version, we arrive at an
equivalent set of axioms which is redundant in the sense that one
of the axioms, namely the second involution law, is derivable from
the other axioms. The set of remaining   axioms is independent.
Finally, it is shown that if both the left-hand and right-hand
distributive laws for relative multiplication are included in the
set of axioms, then two of Tarski's other axioms become redundant,
namely the second involution law and the distributive law for
converse. The  set of remaining axioms is independent and
equivalent to Tarski's axiom system.
\end{abstract}

% --------------------
\maketitle

\section{Introduction} In a series of publications over a period of 40 years, Augustus De\,Morgan\,\cite{dm},
Charles Sanders Peirce (see, in particular, \cite{pe}) and Ernst
Schr\"oder\,\cite{sc}  developed a \textit{calculus of binary
relations} that extended Boole's calculus of classes. At the time
it was considered one of the cornerstones of mathematical logic,
and indeed, in 1903 Bertrand Russell\,\cite{ru} wrote:
\begin{quote} The subject of symbolic logic is formed by three
parts: the calculus of propositions, the calculus of classes, and
the calculus of relations.
\end{quote}  The celebrated theorem of  L\"owenheim\,\cite{low} (which today
would be formulated as stating that every formula   valid in some
model must in fact be valid in some countable model) was proved in
the framework of the calculus of relations.

Interest in the theory gradually dwindled after L\"owenheim's
paper, until 1941, when Alfred Tarski\,\cite{t41} introduced an
abstract algebraic axiomatization of the calculus of relations,
announced several very deep results that he had obtained
concerning the theory, and raised a number of problems that
stimulated interest and research in the subject for decades to
come. Tarski's original axiomatization did not consist entirely of
equations, but he indicated that such an equational axiomatization
was possible (see pp.\,86--87 of \cite{t41}), and by 1943 he had
worked out such an axiomatization (see \cite{t44}). This
equational axiomatization, with minor variations, has subsequently
been used by almost all researchers in the field; see, for
example, Chin-Tarski\,\cite{ct}, Hirsch-Hodkinson\,\cite{hh02},
J\'onsson\,\cite{j82},\cite{j91}, Maddux\,\cite{ma06},
McKenzie\,\cite{mc70}, Monk\cite{mo64}, and
Tarski-Givant\,\cite{tg}, to name but a few. The models of this
set of axioms are called \textit{relation algebras}.

In the early 1940s, J.\,C.\,C.\,McKinsey showed that one of
Tarski's axioms,    the associative law for relative
multiplication, is independent of the remaining axioms of the
theory.  This result was not published at the time, but Tarski
preserved McKinsey's independence model  by presenting it in an
appendix to an unpublished monograph~\cite{t44}, written during
the period 1942--1943.\footnote{McKinsey's independence model is
briefly discussed on pages 357--358 of \cite{ma06}.} For some
time, no further work was done on the question of the independence
of the axioms, but in 1950, Kamel\,\cite{ka52} introduced a rather
different axiomatization and proved   its equivalence to the one
given in Tarski\,\cite{t41}. He established the independence of
some, but not all, of his axioms (see also Kamel\,\cite{ka54}). It
seems that the content of Kamel's work remained
 unknown to   Tarski's school, although there is a reference to \cite{ka54}  in the bibliography of
 \cite{hmt71}.\footnote{Tarski never referred to the paper in his 1970 course on relation algebras
 nor did he mention the paper to Givant during their long collaboration.  As far as we know, with the
 exception of the bibliographic reference in \cite{hmt71}---which finds no echo in the text of \cite{hmt71}---Kamel's work is not referred
 to in any other papers or books originating with members of Tarski's school and dealing with the subject of relation algebras.}

In a graduate topics course on relation algebras that he held at
the University of California at Berkeley in 1970, Tarski talked
briefly about
 McKinsey's result  and  mentioned
that no further work had been done to investigate the independence
of the remaining axioms. The main purpose of the present paper is
to fulfill the goal implicit in Tarski's remark by demonstrating
the independence of all of Tarski's  axioms. The second purpose of
the paper is to show that with a very minor variation in one of
the axioms, namely by using the left-hand form of the distributive
law for relative multiplication instead of the right-hand form,
one of the remaining axioms---namely, the second involution
law---does indeed become derivable from the other axioms and is
therefore not independent of them.  With this redundant law
excluded, the resulting set of axioms is independent and is
equivalent to Tarski's system.  The final purpose of the paper is
to show that if both the left-hand and right-hand distributive
laws for relative multiplication are included in the axiom set,
then two other axioms---namely, the second involution law and the
distributive law for converse---become derivable from the
 set of axioms obtained by excluding these two  laws, and   therefore they are not independent.  With these two redundant laws excluded, the resulting set
of axioms is independent and is equivalent to Tarski's system.

 The original independence models presented in this paper, with
 the exception of those for (R7) and (R9),
were discovered ``by hand", without the aid of a computer.  They
are different from Kamel's models.   Subsequently, a systematic
search, using the model searching program Mace4, developed by
William McCune\,\cite{mcc}, was employed to discover the remaining
two independence models, albeit in very different forms from those
presented here. These models were then analyzed  ``by hand" in
order to understand their true natures and underlying structures.
Mace4 was also employed to determine the minimality of the size of
some of the independence models.

There are several aspects of the paper to which we would like to
draw the readers attention. First, as already mentioned, it
completes the task, begun more than 70 years ago, of showing that
Tarski's axiom system is independent. Second, the independence
models are of some interest in their own right, and may motivate
further study of classes of algebras in which, say, all but one of
Tarski's axioms hold.  The work of Roger Maddux can serve as a
paradigm for such investigations.  In a series of papers
\cite{ma78},  \cite{ma82}, \cite{ma83}, \cite{ma90}, \cite{ma91}
(see also \cite{ma06}), he has studied classes of algebras in
which all of the axioms except the associative law hold, and he
has established interesting and important connections between
these classes of algebras and certain restricted forms of
first-order logic in which there are only three  variables. Third,
although the independence models presented in this paper are
specifically constructed for the purpose of demonstrating the
independence of Tarski's axioms, modifications of these models may
prove useful in establishing independence results for
axiomatizations of other systems of algebraic logic, for example
cylindric algebras and polyadic algebras.  Fourth, future
researchers may find it advantageous to use one of the alternative
axiomatizations of the theory of relation algebras that are
suggested in this paper. They may also find the results in the
paper helpful in determining the role that each of Tarski's axioms
plays in the derivation of various important relation algebraic
laws.  Fifth, as mentioned above, the construction of some of our
independence models has been facilitated by the use of a computer,
and this may resonate with computer scientists who are familiar
with the use of relational methods in computer science.  This
applies in particular to researchers within the RAMiCS community.
Finally, as the referee has pointed out, the results in this paper
may serve as a salutary lesson for readers who believe that the
independence of Tarski's axiom system is not very surprising.
Indeed, as already mentioned, we show that if the right-hand
distributive law for relative multiplication is replaced by what
seems to be a harmless variant, namely the left-hand version of
the law, then one of Tarski's other axioms does turn  out to be
redundant; and if both distributive laws are adopted as axioms,
then two of Tarski's other axioms turn out to be redundant.

We are indebted to Roger Maddux for several remarks that have
helped us to improve the paper, and in particular for suggesting
the use of \refL{lm3.2} (which occurs as part of Theorem 313
in~\cite{ma06}) in order to simplify our original proofs of
Theorems \ref{T:tv} and \ref{T:tv2}.  We are also indebted to the
referee for a very careful reading of the paper, and for several
very helpful suggestions. In particular, he suggested the
inclusion of the remarks in the preceding paragraph, and he also
suggested the current proof of \refL{lm10}, which is much simpler
than our original proof.

\section{Axioms and fundamental notions}\label{S:asec1.1}

Unless explicitly stated otherwise, all algebras below   have the
form
\[\f A=( A\smbcomma +\smbcomma -\smbcomma
;\smbcomma\,\conv\smbcomma\ident\,)\comma   \] where  $A$ is a
non-empty set of elements,  while  $\,+\,$ and $\,;\,$ are binary
operations on $A$, and $\,-\,$ and $\,\conv\,$  are unary
operations on $A$, and $\ident$ is a distinguished constant in
$A$. The set~$A$ is called the \textit{universe} of $\f A$, the
\textit{Boolean} operations $\,+\,$ and~$\,-\,$ are called
\textit{addition} and \textit{complement} respectively, and the
\textit{Peircean} operations $\,;\,$ and~$\,\conv\,$ are called
\textit{relative multiplication} and \textit{converse}
respectively. The distinguished Peircean constant $\ident $ is
called the \textit{identity element}.
  \begin{definition}\label{D:ra} A
\textit{relation algebra} is an algebra $\f A$ in which the
following axioms are satisfied for all elements $r$, $s$,
and $t$ in $\f A$.%\setitemindent{(R10)}
\begin{enumerate} \item [(R1)]  $r+s=s+r$\per \item [(R2)]
$r+(s+t)=(r+s)+t$\per \item [(R3)] $-(-r+s) + -(-r+-s)=r$\per
\item[(R4)]
$r;(s;t)=(r;s);t$\per
\item[(R5)]
$r;\ident=r$\per
\item[(R6)]$r\conv\conv=r$\per
\item[(R7)]
$(r;s)\conv=s\conv;r\conv$\per
\item[(R8)]
$(r+s);t=r;t +s;t$\per
\item[(R9)]
$(r+s)\conv=r\conv +s\conv$\per
\item[(R10)]$r\conv;-(r;s)+
-s=-s$.\qeddef
\end{enumerate}
\end{definition}

 Axiom (R1) is called the \textit{commutative law for addition},
(R2) is the \textit{associative law for addition},  (R3) is
\textit{Huntington's law},  (R4) is the \textit{associative law
for relative multiplication},  (R5) is the \textit{right-hand
identity law for relative multiplication}, (R6) is the
\textit{first involution law},  (R7) is the \textit{second
involution law},  (R8) is the \textit{right-hand distributive law
for relative multiplication}, (R9) is the \textit{distributive law
for converse}, and (R10) is \textit{Tarski's law}.  The
\textit{left-hand distributive law for relative multiplication},
\begin{equation}\tag{R8$'$}\label{Eq:r8}  r;(s+t)=r;s+r;t \comma
\end{equation}
  will also play a role in our  discussion.\comment{When
it is necessary to distinguish between the operations of different
relation
 algebras, we shall occasionally use  superscripts,   as in
\[\f A=( A\smbcomma +^{\f A}\smbcomma -^{\f A}\smbcomma
;^{\f A}\smbcomma\,\conv{}^{\f A}\smbcomma\ident^{\f A}\,)
\]}

 The conventions regarding the
order in which operations are to be performed are as follows:
unary operations take precedence over binary operations, and among
binary operations, multiplications take precedence over additions.
(It is unnecessary in this paper to establish a convention
regarding the order in which different unary operations are to be
performed, as we shall always use parentheses to make this order
clear.) For example, in fully parenthesized form, axioms (R7),
(R8), and~(R10) might be written as
\begin{gather*}
  (r;s)\ssm=(s\ssm);(r\ssm),\qquad (r+s);t=(r;t)+(s;t),\\
  \intertext{and}
((r\conv);(-(r;s)))+ (-s)=-s\per \end{gather*}

Axioms (R1)--(R3) say  that the \textit{Boolean part}  of a
relation algebra~$\f A$, namely the algebra $( A \smbcomma
+\smbcomma -)$, is a Boolean algebra. In particular, the notions
and
 laws\index{law} from the theory of Boolean algebras apply to relation algebras.
 For example, the binary operation $\,\cdot\,$ of \textit{multiplication},
 and the distinguished constants $0$ and $1$ (called \textit{zero} and the \textit{unit} respectively) are defined by
\[r\cdot
s=-(-r+-s),\qquad 1=\ident+-\ident,\qquad
0=-1=-(\ident+-\ident)\per
\]  Similarly, the partial order $\,\le\,$ is defined by
\[r\le s\quad\text{if and only if}\quad r+s=s\per\]   The \textit{supremum}, or \textit{sum}, of a subset $X$ of $A$ is defined to be the least
upper bound of $X$ in the sense of the partial order $\,\le\,$,
but in the case of infinite subsets, such sums may not exist. An
\textit{atom}\index{atom} is defined to be a minimal, non-zero
element, and a Boolean algebra with additional operations is said
to be \textit{atomic} if every non-zero element is above an atom.
The complement of the distinguished constant $\ident$ is called
the \textit{diversity element} and is denoted by $\diver$.  An
atom below $\diver$ is called a \textit{subdiversity atom}.
Whenever some  laws of Boolean algebra are needed to justify a
step in one of the proofs below, we shall simply say that the step
is justified ``\textit{by} (the laws of) \textit{Boolean
algebra}".

Axioms (R4)--(R7) say that the \textit{Peircean part} of a
relation algebra $\f A$, namely the algebra
$(A\smbcomma;\smbcomma\conv\smbcomma\ident)$,  is a monoid with an
 involution that is an anti-isomorphism.  Axioms (R8) and (R9) ensure that
relative multiplication is distributive on the right, and converse
is distributive, over addition.  The binary operation $\,;\,$ is
said to be \textit{completely distributive}, or to
\textit{distribute over arbitrary sums}, if for all subsets $X$
and $Y$ (including infinite subsets and also the empty subset) of
$\f A$, the existence of the sums (or \textit{suprema})~$\tsum X$
and~$\tsum Y$ implies that the sum
\[\tsum\{r;s:r\in X\text{ and }s\in Y\}\] exists and is equal to
$(\tsum X);(\tsum Y)$. A similar definition applies to the
operation of converse.

There are other versions of (R10) that are useful.  For example,
it is clear from the definition of the partial order $\,\le\,$
that if (R1)--(R3) are valid in a model $\f A$, then the validity
of (R10) in $\f A$ is equivalent to the validity of the inequality
\begin{equation}\tag{R10$'$}\label{Eq:r10} r\ssm;-(r;s)\le -s
\end{equation} in $\f A$.
  We shall often make use of this
equivalence in establishing the validity or failure of (R10) in a
model. There is yet another form of (R10) that we shall need.  If
(R1)--(R3) and either \refEq{r8} or else (R6)--(R9) are valid in a
model $\f A$, then the validity of  (R10) in $\f A$ is equivalent
to the validity of the implication
\begin{align*} (r;s)\cdot t=0\qquad&\text{implies}\qquad (r\ssm;t)\cdot
s=0\tag{R11}\label{Eq:r11}\\ \intertext{in $\f A$  (see \refL{lm2}
 and the remark following it in \refS{variant}). If, in addition, the model $\f A$ is
atomic, and if the Peircean operations distribute over arbitrary
sums, then the validity of \refEq{r11} (and hence also of (R10))
is equivalent to the validity of the implication \refEq{r11} for
atoms.  It is often convenient to use the implication \refEq{r11}
in its contrapositive form:} (r\ssm;t)\cdot s\neq
0\qquad&\text{implies}\qquad (r;s)\cdot t\neq 0\per\\
\intertext{In this form, the   version of \refEq{r11} for atoms
assumes the form}
  s\le
r\ssm;t\qquad&\text{implies}\qquad t\le r;s
\end{align*} for all atoms $r$, $s$, and $t$ in $\f A$.

There are two more laws that will play a role in the discussion
below (see \refL{lm1} in \refS{variant}), namely the implications
\begin{alignat*}{3}
r&\le s&\qquad&\text{implies}&\qquad r;t&\le s;t\comma\\
\intertext{and} r&\le s&\qquad&\text{implies}&\qquad t;r&\le
t;s\per
\end{alignat*} They are respectively called the \textit{left-hand} and the \textit{right-hand}\textit{ monotony laws for relative multiplication}.

\comment{In more detail, the left-hand distributive law \refEq{r8}
is derivable from (R6), (R7), and (R9), (see the first part of the
proof of \refT{tv} below). The left-hand distributive law easily
implies the right-hand monotony law for relative multiplication.
Also, the validity of \refEq{r100} for atoms, combined with the
  distributivity of relative multiplication and converse
over arbitrary sums, implies the validity  of \refEq{r100} for
arbitrary elements $r$, $s$, and $t$. Finally,

\cite{g}, \cite{hh02}, or \cite{ma06}).}

\section{Examples of relation algebras}\label{S:models}

The first task of this paper is the construction of   independence
models for each of Tarski's axioms (R1)--(R10).  The models will
often be obtained by taking  well-known relation algebras and
modifying one or more of their operations in some way.  In this
section, we briefly describe the relation algebras that will be
used to construct independence models.

The classic   example  motivating the entire theory  of relation
algebras is the algebra of all binary relations on a set $U$.  The
universe of this algebra is the set of all (binary) relations on
$U$.  The operations of the algebra are union, complement (with
respect to the universal relation $U\times U$), relational
composition, and converse, which are respectively defined by
\begin{gather*}
R\rp S=\{\pair \al\be: \pair\al\ga\in R\text{ and }\pair\ga\be\in
S\text{ for some }\ga\in U\}\\ \intertext{and}
R\mo=\{\pair\al\be:\pair\be\al\in R\}\per
\end{gather*}
 The distinguished constant is the identity relation~$\id U$
on~$U$. The algebra is called the \textit{full set relation
algebra} on $U$.

\begin{table}[htb]
\begin{center} \begin{tabular}{ c | c | c | c | c |} $\,\rp\,$ &
\,$\varnothing$\, & $\id U$ & $\di U$ & \,$U\times U$\,\\ \hline
$\varnothing $ & $\varnothing$ & $\varnothing$ & $\varnothing$ &
$\varnothing$\\ \hline $\id U$ & $\varnothing$ & $\id U$ & $\di U$
& $U\times U$\\  \hline $\di U$ &  $\varnothing$   & $\di U$ &
$U\times U $ & $U\times U $\\ \hline $U\times U$\, & $\varnothing$
& \,$U\times U$\, & $U\times U $ & $U\times U$\\ \hline
\end{tabular}
\end{center} \vspace{.1in}\caption{Relational composition table for $\ff M3$\per}\label{Ta:minsquare}
\end{table}

A more general class of examples  is obtained by allowing the
universe to be an arbitrary set of relations on~$U$ that contains
the universal relation and the identity relation, and that is
closed under the operations of union, complement, relational
composition, and converse. Such algebras, which are called
\textit{set relation algebras}, are subalgebras of the full set
relation algebra on $U$.  For example, fix an arbitrary set $U$ of
cardinality three, and consider the set~$\cc M3$ consisting of the
empty relation $\varnothing$, the identity relation~$\id U$, the
diversity relation $\di U$ (the complement of the identity
relation), and the universal relation $U\times U$. Certainly,~$\cc
M3$ is a subset of the full set relation algebra on $U$, and it
contains the universal relation and the identity relation on $U$.
It is clear that $\cc M3$ is closed under the Boolean operations
of union and complement, and it is equally clear that $\cc M3$ is
closed under the operation of converse, because every relation $R$
in $\cc M3$ is symmetric in the sense that $R\mo=R$. The
relational composition of two relations in $\cc M3$ is again a
relation in $\cc M3$, as \refTa{minsquare} shows, so  $\cc M3$ is
closed under the operation of relational composition. Conclusion:
$\cc M3$ is the universe of a set relation algebra $\ff M3$, and
in fact $\ff M3$ is  the minimal set relation algebra on a set of
cardinality three.

Another class of examples of relation algebras may be constructed
from Boolean algebras. Fix a Boolean algebra $(A\smbcomma
+\smbcomma -)$, define relative multiplication and converse on $A$
to be the operations of Boolean multiplication and the identity
function respectively, and take the identity element to be the
Boolean unit, so that
\[r;s=r\cdot s,\qquad r\conv=r,\qquad\ident =1\] for all $r$ and $s$.  The resulting
algebra $\f A$   is easily seen to be a relation algebra, and  it
is called a \textit{Boolean relation algebra}. A concrete instance
of this construction is provided by the two-element Boolean
algebra, whose universe consists of the elements $0$ and $1$.

A third class of examples of relation algebras may be constructed
from groups. Fix a group
\[( G\smbcomma \scir\smbcomma\inv\smbcomma \ii)\] with a binary
composition operation $\,\scir\,$, a unary inverse operation
$\,\inv\,$, and an identity element~$\ii$.  Take $A$ to be the set
of all subsets of $G$.  Obviously, $A$ is closed under arbitrary
unions and under complements (formed with respect to $G$). Define
operations $\,;\,$ and $\,\conv\,$ of \textit{complex
multiplication} and \textit{complex inverse}  by
\begin{equation*}
  X; Y=\{f\scir g:f\in X\text{ and } g\in
  Y\}\quad\text{and}\quad X\ssm=\{f\mo:f\in X\}
\end{equation*} for all subsets $X$ and $Y$ of $G$, and take the distinguished element $\ident$ to be
the singleton of the group identity element, $\{\ii\}$. The
resulting algebra $\f A$ (in which addition and complement are
defined to be the set-theoretic operations of union and
complement) is a relation algebra, as was shown by McKinsey some
time in the
 1940s (see \cite{jt48}). It is called the \textit{complex
algebra}
  of the group $G$.  For a concrete instance of this construction,
  take $G$ to be the additive group of integers modulo~$3$.
  The operations of relative multiplication and converse in the
  complex
  algebra of this group are set forth in  Tables \ref{Ta:cz3} and \ref{Ta:mz3} respectively.

\begin{table}[htb]
  \centering\begin{tabular}{|c|c|c|c|c|c|c|c|c|}\hline
    % after \\: \hline or \cline{col1-col2} \cline{col3-col4} ...
     $\,;\,$& $\varnot$ & $\{0\}$ & $\{1\}$ & $\{2\}$ & $\{0,1\}$ & $\{0,2\}$ & $\{1,2\}$ & $\{0,1,2\}$ \\ \hline
     $\varnot$& $\varnot$ & $\varnot$ & $\varnot$ &$\varnot$ & $\varnot$ & $\varnot$ & $\varnot$ & $\varnot$ \\ \hline
    $\{0\}$ & $\varnot$ &$\{0\}$ & $\{1\}$ & $\{2\}$ & $\{0,1\}$ & $\{0,2\}$ & $\{1,2\}$ & $\{0,1,2\}$ \\ \hline
    $\{1\}$ & $\varnot$ & $\{1\}$ & $\{2\}$ & $\{0\}$ & $\{1,2\}$ & $\{0,1\}$ & $\{0,2\}$ & $\{0,1,2\}$ \\ \hline
     $\{2\}$ & $\varnot$ & $\{2\}$ & $\{0\}$ & $\{1\}$ & $\{0,2\}$ & $\{1,2\}$ & $\{0,1\}$ & $\{0,1,2\}$ \\ \hline
     $\{0,1\}$ & $\varnot$ & $\{0,1\}$ & $\{1,2\}$ & $\{0,2\}$ & $\{0,1,2\}$ & $\{0,1,2\}$ & $\{0,1,2\}$ & $\{0,1,2\}$ \\ \hline
     $\{0,2\}$& $\varnot$ & $\{0,2\}$ & $\{0,1\}$ & $\{1,2\}$ & $\{0,1,2\}$ & $\{0,1,2\}$ & $\{0,1,2\}$ & $\{0,1,2\}$ \\ \hline
     $\{1,2\}$ & $\varnot$ & $\{1,2\}$ & $\{0,2\}$ & $\{0,1\}$ & $\{0,1,2\}$ & $\{0,1,2\}$ & $\{0,1,2\}$ & $\{0,1,2\}$ \\ \hline
  $\{0,1,2\}$ & \ \ $\varnot$\ \ \  & $\{0,1,2\}$ & $\{0,1,2\}$ & $\{0,1,2\}$ & $\{0,1,2\}$ & $\{0,1,2\}$ & $\{0,1,2\}$ & $\{0,1,2\}$ \\ \hline
 \end{tabular}\vspace{.1in}
  \caption{Relative multiplication table for   the complex algebra of the group of integers modulo $3$.}\label{Ta:cz3}
\end{table}

\begin{table}[htb]
  \centering\begin{tabular}{|c|c|}\hline
    % after \\: \hline or \cline{col1-col2} \cline{col3-col4} ...
     $r$ &   $r\ssm$   \\ \hline
     $\varnot$& $\varnot$   \\ \hline
    $\{0\}$ & $\{0\}$   \\ \hline
    $\{1\}$ & $\{2\}$  \\ \hline
     $\{2\}$ & $\{1\}$  \\ \hline
     $\{0,1\}$ & $\{0,2\}$   \\ \hline
     $\{0,2\}$& $\{0,1\}$ \\ \hline
     $\{1,2\}$ & $\{1,2\}$   \\ \hline
  $\{0,1,2\}$ &  $\{0,1,2\}$  \\ \hline
 \end{tabular}\vspace{.1in}
  \caption{Converse table for the complex algebra of the group of integers modulo $3$.}\label{Ta:mz3}
\end{table}

For the fourth and final example of a relation algebra, consider
the eight-element Boolean algebra~$(D\smbcomma +\smbcomma -)$ with
three atoms, say $\ident$, $a$, and $b$. Define an operation of
relative multiplication on these atoms as in \refTa{b3}, and
extend this operation to all of~$D$ by requiring it to be
distributive over arbitrary sums (see \refTa{b4}, and keep in mind
that $\diver=-\ident=a+b$). Take converse to be the identity
function on $D$. Lyndon\,\cite{l56} was the first to observe the
resulting algebra $\f D$ is a relation algebra. (This algebra is
discussed on page 429 of   Maddux\,\cite{ma06}, where it is
denoted by $\cc 77$; see, in particular, Table 34.)  In fact, $\f
D$ can be represented as a set relation algebra, and also as a
subalgebra of the complex algebra of the group $G\times G$,
where~$G$ is the additive group of integers modulo $3$, but these
observations will not play a role in the discussion below.

\begin{table}[htb]\centering\begin{tabular}{ l | c | c | c|}
\,$\,\,;\,$\, & \,$\ident$\, & \,$a$\, & \,$b$\,\\ \hline $\ident$
& $\ident$ & $a$& $b$\\ \hline $a$ & $a$  & $1$ & $\diver$\\
\hline $b$ & $b$ & $\diver$ & $1$\\ \hline
\end{tabular}\vspace{.1in}\caption{Relative multiplication table
for the atoms in $\f D$}\label{Ta:b3}\end{table}

\begin{table}[htb]
  \centering\begin{tabular}{|c|c|c|c|c|c|c|c|c|}\hline
    % after \\: \hline or \cline{col1-col2} \cline{col3-col4}
     $\,;\,$& $0$ & $\ident$ & $a$ & $b$ & $\ident+a$ & $\ident +b$ & $\diver$ & $1$ \\ \hline
     $0$& $0$ & $0$ & $0$ &$0$ & $0$ & $0$ & $0$ & $0$ \\ \hline
    $\ident$ &  $0$ & $\ident$ & $a$ & $b$ & $\ident+a$ & $\ident +b$ & $\diver$ & $1$\\ \hline
    $a$ & $0$ & $a$ & $1$ & $\diver$ & $1$ & $\diver$ & $1$ & $1$ \\ \hline
     $b$ & $0$ & $b$ & $\diver$ & $1$ & $\diver$ & $1$ & $1$ & $1$ \\ \hline
     $\ident+a$ & $0$ & $\ident +a$ & $1$ & $\diver$ & $1$ & $1$ & $1$ & $1$ \\ \hline
     $\ident + b$& $0$ & $\ident+b$ & $\diver$ & $1$ & $1$ & $1$ & $1$ & $1$ \\ \hline
     $\diver$ & $0$ & $\diver$ & $1$ & $1$ & $1$ & $1$ & $1$ & $1$ \\ \hline
  $1$ & \ \ $0$\ \ \  & $1$ & $1$ & $1$ & $1$ & $1$ & $1$ & $1$ \\ \hline
 \end{tabular}\vspace{.1in}
  \caption{Relative multiplication table for   $\f D$.}\label{Ta:b4}
\end{table}

\section{Independence}\label{S:theorem}

A mathematical statement $\vp$ is said to be \textit{independent}
of a  set  of mathematical statements $\Phi$ (with respect to a
given logical framework) if $\vp$ cannot be derived from~$\Phi$
(within the given logical framework).   A set
 of axioms $\Phi$ is said to be \textit{independent} if each $\vp$
in $\Phi$ is independent of the set of axioms  obtained from
$\Phi$ by removing $\vp$. In other words, $\Phi$ is independent if
none of the axioms in $\Phi$ can be derived from the remaining
axioms in~$\Phi$.  The standard way of establishing the
independence of a set of axioms is to construct for each axiom
$\vp$, a model in which~$\vp$ fails and the remaining axioms are
valid. Such a model is called an \textit{independence model} for
$\vp$. The first task  of the present paper is to prove the
following theorem.
\begin{theorem}\label{T:theorem}
 The set of Tarski's axioms \textnormal{(R1)--(R10)} is independent\per
\end{theorem}
The proof proceeds by constructing for each $n=1,\dots,10$ an
independence model $\ff An$ for~(R$n$).

\section{Independence of (R1)}\label{S:r1}

Let $(G\smbcomma \scir \smbcomma{}\mo\smbcomma \ii)$ be any
Boolean group of order at least two, with identity element $\ii$,
that is to say, any group with at least two elements in which each
element $r$ is its own inverse, so that $r\mo=r$.    Let addition
be the binary operation of left-hand projection on $G$, which is
defined by
\[r+s=r\] for all $r $ and $s$ in $G$,  take complement  to be the identity operation  on
$G$, and take relative multiplication, converse, and the identity
element to coincide with the corresponding group operations and
identity element, so that
\[r;s=r\scir s,\qquad r\ssm=r\mo=r,\qquad\text{and}\qquad
\ident=\ii\] for all $r$ and $s$ in $G$. In the resulting algebra
$\ff A1$ (of the same similarity type as relation algebras), it is
clear that (R1) fails. Indeed, for distinct elements $r$ and $s$,
we have
\[r+s=r\neq s=s+r\per\]

The sum of any finite sequence of elements in $\ff A1$ is always
the left-most element in the sequence, by the definition of
addition, so the associative law (R2)  holds automatically; in
more detail,
\[r+(s+t)=r=r+t=(r+s)+t\per\] Also Huntington's law (R3) holds:
\[-(-r+s)+-(-r+-s)=(r+s)+(r+s)=r\comma\] by the definitions of
complement and addition.  Axioms (R4)--(R7) hold trivially in~$\ff
A1$, because the Peircean part of $\ff A1$ is, by definition, a
Boolean group, and Boolean groups are always commutative. The
distributive law (R8) holds in $\ff A1$, because
\[(r+s);t=r;t=r;t+s;t\comma\] by the definition of addition.  The
distributive law (R9)  holds trivially in $\ff A1$, since converse
is the identity operation.  Also, Tarski's law  (R10) holds in
$\ff A1$, because
\begin{multline*}
r\ssm;-(r;s)+-s=r\ssm;(r;s)+s=r\ssm;(r;s)\\=r\mo\scir(r\scir s)
=(r\mo\scir r)\scir s=\ii\scir s =s=-s\per
\end{multline*}
These equalities use the fact that complement is the identity
operation on $G$ and addition is the left-hand projection, while
converse, relative multiplication, and the identity element in
$\ff A1$ coincide with the corresponding group operations and
identity element, by definition.

A minimal independence model for (R1) is obtained by starting with
the two-element additive Boolean group of integers modulo $2$.
\section{Independence of (R2)}\label{S:r2}

The independence model $\ff A2$ is defined as follows.  The
universe consists of three elements, $0$, $\ident$, and $1$, and
 the operations of addition, relative multiplication, and
complement are determined as in \refTa{a2}. Converse is defined to
be the identity function on the universe, and the identity element
is defined to be $\ident$.
\begin{table}[htb]\centering\begin{tabular}{l | c | c | c|}
$+$ & \,$0$\, & \,$\ident$\, & \,$1$\,\\ \hline $0$ & $0$ &
$\ident$& $1$\\ \hline $\ident$ & $\ident$  & $\ident$ & $0$\\
\hline $1$ & $1$ & $0$ & $1$\\ \hline
\end{tabular}\qquad\qquad
\begin{tabular}{l | c | c | c|}
$;$ & \,$0$\, & \,$\ident$\, & \,$1$\,\\ \hline $0$ & $0$ & $0$&
$0$\\ \hline $\ident$ & $0$  & $\ident$ & $1$\\ \hline $1$ & $0$ &
$1$ & $\ident$\\ \hline
\end{tabular}
\qquad\qquad
\begin{tabular}{ l | c |}
$r$ &  $-r$\\ \hline $0$ & $0$\\ \hline $\ident$ & $1$\\ \hline
$1$ & $\ident$\\ \hline
\end{tabular}\vspace{.1in}\caption{Addition, relative multiplication, and complement tables for the
algebra $\ff A2$.}\label{Ta:a2}\end{table} The associative law
(R2)  is easily shown to fail in $\ff A2$: just take $r$, $s$, and
$t$ to be $\ident$, $\ident$, and~$1$ respectively to arrive at
\begin{align*}
  r+(s+t)&=\ident+(\ident+1)=\ident +
  0=\ident\\ \intertext{and} (r+s)+t&=(\ident+\ident)+1=\ident
  +1=0\per
\end{align*}

Turn now to the task of verifying   the remaining
 axioms in $\ff A2$.  Axioms (R6) and~(R9) are
 valid because converse is defined to be the identity
function; and (R5) is clearly valid, as a glance at the column for
$\ident$ in the relative multiplication table for $\ff A2$ shows
(see \refTa{a2}). Axiom (R1) is valid because the operation table
for addition is symmetric across the diagonal, and therefore
addition is commutative. Similarly,~(R7) is valid because the
operation table for relative multiplication is symmetric across
the diagonal---so that relative multiplication is
commutative---and converse is the identity function.  It remains
to check (R3), (R4), (R8), and~(R10).

Begin with the verification of (R3). If $r$ is $0$, then
\begin{multline*}
-(-r+s)+-(-r+-s)=-(-0+s)+-(-0+-s)\\=-(0+s)+-(0+-s)=-s+-(-s)=-s+s=0=r\comma
\end{multline*}
by the definitions of complement and addition, and the fact that
the sum of any element and its complement  is always $0$ in $\ff
A2$ (see \refTa{a2}). Similarly, if $s$ is $0$, then
\begin{multline*}
  -(-r+s)+-(-r+-s)=-(-r+0)+-(-r+-0)\\=-(-r)+-(-r)=r\per
\end{multline*}  Assume now that $r$ and $s$ are both
non-zero. If $r=s$, then
\begin{equation*}
-(-r+s)+-(-r+-s)=-(-r+r)+-(-r+-r)=-0+-(-r)=r\per
\end{equation*} The second equality uses two properties of addition in $\ff A2$: it is an idempotent operation in the sense that $t+t=t$
for all $t$; and the sum of an element and its complement is
always $0$. If $r$ and $s$ are distinct, then  $r=-s$ (because
$\ff A2$ has just two non-zero elements, and they  are the
complements of one another), and consequently (R3) is valid for
$r$ and $-s$, by the case  just considered. It follows that
\begin{multline*}
-(-r+s)+-(-r+-s)=-(-r+-s)+-(-r+s)\\=-(-r+-s)+-(-r+-(-s))=r\comma
\end{multline*}
by the commutativity of addition, the fact that $-(-s)=s$, and the
validity of (R3) for $r$ and $-s$.

As regards the associative law (R4), if at least one of the
elements $r$, $s$, and $t$ is~$0$, then both sides of (R4) reduce
to $0$, by the definition of relative multiplication. Similarly,
if at least one of the three elements is $\ident$, then both sides
of (R4) reduce to the relative product of the other two elements.
For example, if $s$ is $\ident$, then
\[r;(s;t)=r;(\ident;t)=r;t\qquad\text{and}\qquad
(r;s);t=(r;\ident);t=r;t\per\]The only other possibility is that
all three elements are $1$, and in this case both sides of (R4)
reduce to $1$.

Turn next to the verification of (R8).  If $t$ is $0$, then both
sides of (R8) reduce to $0$, and if $t$ is $\ident$, then both
sides   reduce to $r+s$, by the definition of relative
multiplication. Similarly, if $r$ is $0$, then both sides of (R8)
reduce to $s;t$, and if~$s$ is $0$ or if $r=s$, then both sides
reduce to $r;t$. There remain the cases when~$t$ is $1$, and $r$
and $s$ are distinct values in the set $\{\ident,1\}$. In this
case, $r+s$ is $0$, by the definition of addition, so the left
side of (R8) reduces to $0$; and the right side of~(R8) reduces to
$\ident;1+1;1$, which is also $0$.

To verify (R10),  observe first that $1;s=-s$ for any value of $s$
in $\ff A2$, by the definitions of relative multiplication and
complement (see \refTa{a2}). If $r$ is $0$, then
\begin{multline*}
r\ssm;-(r;s) +-s=0\ssm;-(0;s)+-s=0;-0+-s\\=0;0+-s=0+-s=
-s,\end{multline*}if $r$ is $\ident$, then
\begin{gather*} r\ssm;-(r;s)+-s=\ident\conv;-(\ident;s)+-s=\ident;-s+-s=-s\comma\\\intertext{and
if $r$ is $1$, then}
 r\ssm;-(r;s)+-s=1\conv;-(1;s)+-s=1;-(-s)+-s=1;s+-s=-s\per
\end{gather*}
   Thus, in all
three cases, the left side of (R10) reduces to $-s$, as desired.

A computational check using  the model searching program Mace4 has
shown that~$\ff A2$ is the unique independence model for (R2) of
cardinality $3$, and  there is no independence model for (R2) of
smaller cardinality.

\section{Independence of (R3)}\label{S:r3}

Consider any Boolean relation algebra with at least two elements.
Modify the definition of complement in this algebra by requiring
it to be the identity function, that is to say, by
requiring~$-r=r$ for all $r$.    In the resulting algebra~$\ff
A3$, axiom~(R3) fails, because
\[-(-r+s)+-(-r+-s)=r+s\comma\] by the definition of complement, and
$r+s$  is different from $r$ whenever the element~$s$ is not below
$r$ (as is the case if, say, $s$ is $1$ and $r$ is less than~$1$).
Axioms (R1),~(R2), and~(R4)--(R9) all hold in $\ff A3$, because
they do not involve the operation of complement. Also, (R10) holds
in $\ff A3$, because
\[r\ssm;-(r;s)+-s=r\cdot(r\cdot s) + s=s=-s\comma\] by the
definitions of the operations of converse, relative
multiplication, and complement.

A minimal independence model for (R3) is obtained by starting with
a Boolean relation algebra of cardinality $2$.

\section{Independence of (R4)}\label{S:r4}  To construct
an independence model for (R4),  start with a three-element
partial algebra $(G\smbcomma \scir\smbcomma{}\mo\smbcomma \ii)$
of the same similarity type as a group. \begin{table}[htb]
  \centering\begin{tabular}{|c|c|c|c|}\hline
    % after \\: \hline or \cline{col1-col2} \cline{col3-col4} ...
   \  $\,\scir\,$ \ & \ $0$ \  & \ $1$ \  & \ $2$ \ \\ \hline
     $0$  &$0$ & $1$ & $2$ \\ \hline
     $1$ & $1$ & $0$ &\    \\ \hline
     $2$   & $2$ & \  & $0$ \\ \hline
  \end{tabular}\vspace{.1in}
  \caption{Table for the operation $\,\scir\,$.}\label{Ta:mckinsey2}
\end{table}   The universe $G$ of this partial algebra is the
set $\{0,1,2\}$, the binary partial operation $\scir$ is
determined by \refTa{mckinsey2}, the unary operation $\,{}\mo\,$
is  the identity function on $G$, and the distinguished constant
$\ii$ is~$ 0$. The values of $1\scir 2$ and $2\scir 1$ in the
operation table for $\,\scir\,$ are left undefined. Form the
complex algebra $\ff A4$ of this partial algebra in exactly the
same way as the complex algebras of groups are formed. The
operation of relative multiplication in $\ff A4$ is given by
\refTa{mckinsey1}. Converse is the identity function on the
universe of $\ff A4$, and $\{0\}$ is the identity element with
respect to the operation of relative multiplication.
\begin{table}[htb]
  \centering\begin{tabular}{|c|c|c|c|c|c|c|c|c|}\hline
    % after \\: \hline or \cline{col1-col2} \cline{col3-col4} ...
     $\,;\,$& $\varnot$ & $\{0\}$ & $\{1\}$ & $\{2\}$ & $\{0,1\}$ & $\{0,2\}$ & $\{1,2\}$ & $\{0,1,2\}$ \\ \hline
     $\varnot$& $\varnot$ & $\varnot$ & $\varnot$ &$\varnot$ & $\varnot$ & $\varnot$ & $\varnot$ & $\varnot$ \\ \hline
    $\{0\}$ & $\varnot$ &$\{0\}$ & $\{1\}$ & $\{2\}$ & $\{0,1\}$ & $\{0,2\}$ & $\{1,2\}$ & $\{0,1,2\}$ \\ \hline
    $\{1\}$ & $\varnot$ & $\{1\}$ & $\{0\}$ & $\varnot$ & $\{0,1\}$ & $\{1\}$ & $\{0\}$ & $\{0,1\}$ \\ \hline
     $\{2\}$ & $\varnot$ & $\{2\}$ & $\varnot$ & $\{0\}$ & $\{2\}$ & $\{0,2\}$ & $\{0\}$ & $\{0,2\}$ \\ \hline
     $\{0,1\}$ & $\varnot$ & $\{0,1\}$ & $\{0,1\}$ & $\{2\}$ & $\{0,1\}$ & $\{0,1,2\}$ & $\{0,1,2\}$ & $\{0,1,2\}$ \\ \hline
     $\{0,2\}$& $\varnot$ & $\{0,2\}$ & $\{1\}$ & $\{0,2\}$ & $\{0,1,2\}$ & $\{0,2\}$ & $\{0,1,2\}$ & $\{0,1,2\}$ \\ \hline
     $\{1,2\}$ & $\varnot$ & $\{1,2\}$ & $\{0\}$ & $\{0\}$ & $\{0,1,2\}$ & $\{0,1,2\}$ & $\{0\}$ & $\{0,1,2\}$ \\ \hline
  $\{0,1,2\}$ & \ \ $\varnot$\ \ \  & $\{0,1,2\}$ & $\{0,1\}$ & $\{0,2\}$ & $\{0,1,2\}$ & $\{0,1,2\}$ & $\{0,1,2\}$ & $\{0,1,2\}$ \\ \hline
 \end{tabular}\vspace{.1in}
  \caption{Operation table for relative multiplication in $\ff A4$.}\label{Ta:mckinsey1}
\end{table}

  To see that
(R4) fails in $\ff A4$, take $r$ to be the atom $\{1\}$, and take
$s$ and $t$ to be the atom $\{2\}$,  to obtain
\begin{align*}
r;(s;t)&=\{1\};(\{2\};\{2\})=\{1\};\{0\}=\{1\},\\ \intertext{and}
(r;s);t&=(\{1\};\{2\});\{2\}=\varnot;\{2\}=\varnot\per
\end{align*}

The Boolean part of $\ff A4$ is, by definition, a Boolean algebra
of sets, so  (R1)--(R3) are certainly valid in $\ff A4$. The
operation of relative multiplication is commutative and
distributive over addition in $\f A$, because $\,\scir\,$ is a
commutative partial operation (see \refTa{mckinsey2}), and because
the very definition of the complex operation $\,;\,$ in terms of
$\,\scir\,$ (see \refS{models}) implies that it  distributes over
arbitrary sums. From these observations, together with the fact
that converse is the identity function on the universe of $\ff
A4$, and $\{0\}$ is an identity element with respect to operation
of relative multiplication, it follows that  (R5)--(R9) all hold
trivially in $\ff A4$.

It remains to show that (R10) is valid in $\ff A4$. We do this by
verifying condition~\refEq{r11} for atoms in its contrapositive
form (see the remarks at the end of \refS{asec1.1}). In the
present situation, this amounts to checking that
\begin{equation*}\tag{1}\label{Eq:r4.1}
s\le r;t\qquad\text{implies}\qquad t\le r;s
\end{equation*}
  for all atoms $r$, $t$,
and $s$. If $r$ is the identity element $\{0\}$, then \refEq{r4.1}
reduces to the triviality  that $s=t $ implies $t=s$. If $t$ is
the identity element, then the hypothesis of~\refEq{r4.1} reduces
to  $s=r$;  in this case $r;s=r;r$, which is always the identity
element when $r$ is an atom (see \refTa{mckinsey1}), so the
conclusion of \refEq{r4.1} holds. We may therefore assume that~$r$
and~$t$ are atoms distinct from the identity element.  If $r=t$,
then the hypothesis of \refEq{r4.1} is only satisfied if $s$ is
the identity element (see \refTa{mckinsey1}), and in this case the
conclusion of \refEq{r4.1} holds trivially. The only remaining
case is when~$r$ and $t$ are, in some order, the two subdiversity
atoms $\{1\}$ and $\{2\}$. In this case the relative product $r;t$
is the empty set (see \refTa{mckinsey1}), so the hypothesis of
\refEq{r4.1} is never satisfied, and therefore the implication in
\refEq{r4.1} is always true.

The algebra $\ff A4$ was discovered   by J.\,C.\,C.\,McKinsey some
time in the early 1940s.  A computational check using Mace4 has
shown that there is no independence model for (R4) of smaller
cardinality.  Roger Maddux has studied variants of relation
algebras in which only weakened versions of the associative law
hold. In particular, he has constructed numerous examples of
algebras in which (R4)  fails and the rest of Tarski's axioms
hold; see in particular Theorems 2.5(3),  2.5(4), 3.7,  and 3.10
in~\cite{ma82} and see also \cite{ma06}.

\section{Independence of (R5)}\label{S:r5}

 Consider any Boolean   algebra $(A\smbcomma
+\smbcomma -)$ with at least two elements.   Define relative
multiplication to be the binary operation on $A$ whose value on
any two arguments is always $0$, so that
\[r;s=0\] for all $r$ and $s$. Take converse to be the
identity function on $A$, and take $\ident$ to be any element in
$A$. In the resulting algebra $\ff A5$,
\[r;\ident=0\neq r\] whenever $r$ is a non-zero element, so (R5) fails. The
Boolean axioms (R1)--(R3) obviously hold in $\ff A5$, and the
associative law (R4) for relative multiplication holds because
both sides of (R4) reduce to $0$.  The same is true of (R7),
\[(r;s)\ssm=0\ssm=0=s\ssm;r\ssm\comma\] and (R8),
\[(r+s);t=0=0+0=r;t+s;t\per\] Axioms (R6) and (R9) hold trivially
because converse is the identity function, and Tarski's law (R10)
holds because \[r\ssm;-(r;s)+-s=0+-s=-s\per\]

There is another, rather trivial independence model for (R5) that
should be mentioned. Consider any relation algebra $\f B$ with at
least two elements. The identity element in $\f B$ is uniquely
determined in the sense that there is exactly one element in $\f
B$ for which (R5) holds (see Theorem 1.2 in
Chin-Tarski\,\cite{ct}). Take $\ff B5$ to be the algebra obtained
from $\f B$ by choosing $\ident$ to be any element different from
the identity element in~$\f B$ (for example, choose $\ident$ to be
the zero element in $\f B$). Axiom (R5)  fails in~$\ff B5$,
because $\ident$ is not the identity element in $\f B$. But
(R1)--(R4) and (R6)--(R10)  all hold in~$\ff B5$, because they
hold in $\f B$ and they do not explicitly mention $\ident$.

The independence model $\ff B5$ has one important defect. There is
a formulation  of~(R5) that is not equational, but rather
existential in form and   does not utilize a distinguished
constant; instead, it asserts the existence of a right-hand
identity element for relative multiplication (see, for example,
Chin-Tarski\,\cite{ct}).  This existential form of~(R5) is true in
$\ff B5$, so $\ff B5$ cannot be used to demonstrate the
independence of the existential form of (R5) from the remaining
 axioms.  On the other hand,~$\ff A5$ can still be used for this purpose.

A minimal independence model for (R5) may be obtained by using the
two-element Boolean algebra to construct  $\ff A5$.

\section{Independence of (R6)}\label{S:r6}  Let $\f A$ be any
relation algebra with at least two elements, and modify the
definition of converse by requiring
\[r\conv=0\] for all $r$. Obviously, (R6) fails in the resulting algebra
$\ff A6$, since  for any non-zero element $r$ we have
\[r\ssm\conv=0\neq r\per\]

It is equally clear (R1)--(R5) and (R8) are valid in   $\ff A6$,
because these axioms are valid in $\f A$ and do not contain any
occurrence of converse.   Also, (R7) is valid in $\ff A6$, because
\[(r;s)\ssm=0=0;0=s\ssm;r\ssm\comma\] by the definition of
converse, and the fact that relative multiplication by $0$ always
yields $0$ in the relation algebra $\f A$, and hence also in $\ff
A6$.  A similar argument shows that (R9) is valid in $\ff A6$,
because both sides of the axiom reduce to $0$. Finally, to verify
(R10), or equivalently, \refEq{r10}, in~$\ff A6$, observe that
\[r\ssm;-(r;s)=0;-(r;s)=0\le-s\comma\] by the definition of
converse, and the fact that relative multiplication by $0$ always
yields $0$ in $\f A$, and hence also  in $\ff A6$.

A minimal independence model for (R6) may be obtained by taking
$\f A$ to be the two-element Boolean  relation algebra.

\comment{
 Let $(A\smbcomma +\smbcomma -)$ be a four element Boolean
algebra with zero $0$, unit~$1$, and atoms $\ident$ and $\diver$.
Take $\,\conv\,$ to be the constant unary function on $A$ that
maps every element to $0$, take $\,;\,$ to be the commutative
binary operation on $A$ determined by \refTa{R6fails}, and put
\[\f A=(A\smbcomma +\smbcomma -\smbcomma ;\smbcomma
\conv\smbcomma\ident)\per\]

 Notice that \refTa{R6fails} is \begin{table}[ht]
\begin{center} \begin{tabular}{ c | c | c | c | c |} $\,;\,$ &
\,$0$\, & $\ident$ & $\diver$ & \,$1$\,\\ \hline $0 $ & $0$ & $0$
& $0$ & $0$\\ \hline $\ident$ & $0$ & $\ident$ & $\diver$ & $1$\\
\hline $\diver$ & $0$ & $\diver$ & $1$ & $1$\\ \hline $1$\, & $0$
& \,$1$\, & $1$ & $1$\\ \hline
\end{tabular} \end{center} \caption{Table for the operation $\,;\,$\per}\label{Ta:R6fails}\end{table}
just the relative multiplication table of the minimal set relation
algebra $\ff M3$ (see \refS{asec01.01}). Since the relation
algebraic axioms are true in $\ff M3$,  every axiom that does not
involve the computation of a converse must be true in $\f A$.  In
particular, (R1)--(R5) and (R8) are true in $\f A$.  Also, (R7)
and (R9) are true in $\f A$, because the two sides of each
equation  always evaluate to $0$. As regards (R10), since $r\ssm$
is $0$, by the definition of $\,\conv\,$, we have
\[r\ssm;-(r;s)=0;-(r;s)=0\per\]  Consequently, the two sides of
(R10) always evaluate to $-s$, so (R10) is true in $\f A$.
However, (R6) fails in $\f A$, because $1\ssm\conv$ is $0$, not
$1$.}

\section{Independence of (R7)}\label{S:r7}

Let $(A\smbcomma +\smbcomma -)$ be any Boolean algebra with at
least four elements, and take $\ident$ to be any element in $A$
that is different from  $0$ and $1$. For instance, $\ident $ might
be an atom. Define a binary operation $\,;\,$ on $A$ by modifying
slightly the definition of relative multiplication in the
independence model $\ff A5$ for (R5):
\begin{equation*}%\tag{}\label{Eq:}
  r;s=\begin{cases}
   r &\quad \text{if\quad $s=\ident$}\comma \\
   0 &\quad \text{if\quad $s\neq\ident$}\comma
 \end{cases}
\end{equation*} for all $r$ and $s$.  Take $\,\conv\,$
to be the identity function on $A$.

Axiom (R7) fails in the resulting algebra $\ff A7$  because
relative multiplication is not commutative.  In more detail, if
$r$ is an element different from $0$ and $\ident$, and if~$s$
is~$\ident$,  then
\[(r;s)\ssm=r;\ident=r\neq 0=\ident;r=s\ssm;r\ssm\per \]

To see that (R4) is valid in $\ff A7$, observe that if $t=\ident$,
then both sides of (R4) reduce to $r;s$, and if $t\neq\ident$,
then both sides of (R4) reduce to $0$, by the definition of
relative multiplication. The argument that (R8) holds is similar:
if $t=\ident$, then both sides of (R8) reduce to $r+s$, and if
$t\neq \ident$, then both sides reduce to $0$. Each of~(R5),~(R6),
and (R9) holds trivially   in $\ff A7$, by the definitions of
relative multiplication and converse. The verification of (R10),
in the form of \refEq{r10},  in~$\ff A7$ breaks into cases, and
the argument in each case is based on the definition of converse
and relative multiplication, but uses also laws of Boolean algebra
in the final step. If~$s\neq \ident$, then
\[r\ssm;-(r;s)=r;-0=r;1=0\le -s\per\]
If $s=\ident$, then either $r\neq \diver$, in which case
$-r\neq\ident$, and therefore
\[r\ssm;-(r;s)=r;-(r;\ident)=r;-r=0\le-s\comma\] or else
$r=\diver$, in which case $-r=\ident$, and therefore
\[r\ssm;-(r;s)=\diver;-(\diver;\ident)=\diver;-\diver=\diver;\ident=\diver=-\ident=-s\per\]

A computational check using   Mace4 has shown that there is no
independence model for (R7) of smaller cardinality.

\section{Independence of (R8)}\label{S:r8}
 Let $\f A$ be any symmetric, integral relation algebra, that is to say, any relation
 algebra with at least two elements in which
 \[r;s=0\qquad\text{implies}\qquad r=0\quad\text{or}\quad s=0\comma\] and in which converse is the identity
 function on $A$.   It is well known and easy to see that the operation of relative multiplication in such a relation algebra
 is commutative, and  for any
 non-zero element  $r$,
 \[r;1=1;r=1\] (see \cite{jt52}).  The independence model   $\ff A8$ is obtained from $\f A$ by changing
 the definition of relative multiplication in one instance, namely  when
 both arguments are $0$, and in this case putting
 \[0;0=1\per\]
It is not difficult to check that (R8) fails in $\ff A8$: just
take $r$ and $t$ to be~$0$, and $s$ to be $\ident$, and observe
that
\begin{gather*}
(r+s);t=(0+\ident);0=\ident;0=0, \\ \intertext{but}
r;t+s;t=0;0+\ident;0=1+0=1\per
\end{gather*}

  The axioms of relation algebra are
valid in $\f A$, by assumption.  Every instance of an axiom that
does not involve a computation of~$0;0$ yields the same result in
$\ff A8$ as it does in $\f A$, so it must hold in $\ff A8$. In
particular, (R1)--(R3), (R5), (R6), and~(R9) all hold in $\ff A8$.
Similarly, all instances of (R7) in which $r$ and~$s$ are not
both~$0$
    hold in $\f A$ and therefore in $\ff A8$; and when both $r$ and $s$ are~$0$, each
   side
   of~(R7) reduces to $1$. Thus, (R7) is valid in $\ff A8$.
     It remains to check the validity of (R4) and (R10).

Every instance  of (R4) in which at most one of~$r$,~$s$, and~$t$
is $0$ must hold in $\ff A8$, since no such instance can involve a
computation of $0;0$. (Here, the assumption that~$\f A$ is
integral plays a role.) If all three of these elements are~$0$,
then the computations
\begin{align*}
0;(0;0)=0;1=0\qquad&\text{and}\qquad (0;0);0=1;0=0\\
\intertext{show  that (R4) holds in $\ff A8$ in this case as well.
There remain the three cases when exactly two of the elements are
$0$. If $r$ and $s$ are both $0$, and $t$ is different from $0$,
then} r;(s;t)=0;0=1\qquad&\text{and}\qquad (r;s);t=1;t=1\\
\intertext{(the assumption that $\f A$ is integral justifies the
last step).  A similar argument applies when $s$ and $t$ are $0$,
and $r$ is different from~$0$. If $r$ and $t$ are $0$, and $s$ is
different from $0$, then} r;(s;t)=0;0=1\qquad&\text{and}\qquad
(r;s);t=0;0=1\per
\end{align*}
Conclusion: (R4) is  valid in $\ff A8$.

There is only one instance of (R10) that involves a computation
of~$0;0$, namely  when $r$ and $s$ are both $0$. In this case, and
in every other case in which~$s$ is $0$, we have $-s=1$, so the
two sides of (R10) evaluate to $1$.  Since the remaining instances
of (R10)  do not lead to a computation involving $0;0$,   they
automatically hold in~$\ff A8$. For example, suppose $r$ and $s$
are not both $0$, but $-(r;s)$ is $0$.  In this case,~$r;s$ must
be $1$, so $r$---and therefore also $r\ssm$---must be different
from $0$. Consequently, this instance of (R10) does not involve a
computation of $0;0$, and therefore it holds in $\ff A8$.

Take $\f A$ to be the two-element Boolean relation algebra   to
arrive at a minimal independence model for (R8).

\section{Independence of (R9)}

Let $\f C$ be the complex algebra  of the additive group of
integers modulo $3$ (see \refS{asec1.1}, and in particular
\refTa{cz3}). In order to avoid notational confusion,
write~$\,\scir\,$ and $\,{}\mo\,$ for the operations of relative
multiplication (composition of complexes) and converse (inversion
of complexes) in $\f C$, and write~$\,;\,$ and $\,\conv\,$ for the
corresponding operations in $\ff A9$.
 The independence model $\ff A9$ for
(R9) is obtained from $\f C$ by changing the definitions of
converse and relative multiplication slightly, while leaving the
remaining operations intact.  In fact, the table for converse
  in $\ff A9$ is obtained from the table for converse   in $\comz$
by changing the value of converse on the two singletons~$\{1\}$
and $\{2\}$, while leaving its value on the remaining elements
unchanged. In $\comz$, converse interchanges these two singletons,
whereas in $\ff A9$ converse is defined to map each of these
singletons to itself. Put somewhat differently, converse in~$\ff
A9$ maps every element to itself, with the exception of the two
doubletons~$\{0,1\}$ and~$\{0,2\}$, which it interchanges. As a
result,  (R9) must fail in $\ff A9$. Indeed, if $r$ and $s$ are
taken to be $\{0\}$ and $\{2\}$ respectively, then
\begin{equation*}
(r+s)\ssm=(\{0\}+\{2\})\ssm=\{0,2\}\ssm=\{0,1\}\neq\{0\}+\{2\}=r\ssm+s\ssm\per
\end{equation*}

Unfortunately, this change in the definition of converse causes
other axioms to fail, for example (R7).  In order to avoid this
undesired side effect, the operation of relative multiplication
must also be altered in the passage from $\comz$ to $\ff A9$.
Specifically, it is altered in the six cases that involve relative
multiplication of one of the two singletons $\{1\}$ and $\{2\}$ on
the left with one of the three doubletons on the right.  If $r$ is
any singleton, and $s$ any doubleton, then the relative product
$r;s$ in $\ff A9$ is defined to coincide with the relative product
of $r\mo$ and $s$ in $\f C$, in symbols
\[r;s=r\mo\scir s\per\]
This has the effect of interchanging the relevant parts of the
rows for $\{1\}$ and $\{2\}$ in the operation table for relative
multiplication in $\f C$, but   leaving the row for $\{0\}$
unchanged (see Tables~\ref{Ta:cz3} and \ref{Ta:mmr9}).
\begin{table}[htb]
  \centering\begin{tabular}{|c|c|c|c|}\hline
    % after \\: \hline or \cline{col1-col2} \cline{col3-col4} ...
     $\,\scir\,$& $\{0,1\}$ & $\{0,2\}$ & $\{1,2\}$ \\ \hline
       $\{0\}$ & $\{0,1\}$ & $\{0,2\}$ & $\{1,2\}$\\ \hline
        $\{1\}$ & $\{1,2\}$ & $\{0,1\}$ & $\{0,2\}$
  \\ \hline $\{2\}$ &
  $\{0,2\}$ & $\{1,2\}$ & $\{0,1\}$  \\ \hline
 \end{tabular}\qquad\qquad
  \begin{tabular}{|c|c|c|c|}\hline
    % after \\: \hline or \cline{col1-col2} \cline{col3-col4} ...
     $\,;\,$& $\{0,1\}$ & $\{0,2\}$ & $\{1,2\}$ \\ \hline
       $\{0\}$ & $\{0,1\}$ & $\{0,2\}$ & $\{1,2\}$\\ \hline
        $\{1\}$ & $\{0,2\}$ & $\{1,2\}$ & $\{0,1\}$
  \\ \hline $\{2\}$ &
  $\{1,2\}$ & $\{0,1\}$ & $\{0,2\}$  \\ \hline
 \end{tabular}\vspace{.1in}
  \caption{Comparison of the  differences in the relative multiplication tables for $\comz$ and  $\ff A9$.}\label{Ta:mmr9}
\end{table}
In particular, relative multiplication by $\{0\}$  in $\ff A9$
yields the same result as in $\f C$.

The Boolean axioms (R1)--(R3) obviously hold in $\ff A9$, because
the Boolean part of $\ff A9$ coincides with the Boolean part of
$\f C$. Similarly, the identity law (R5) and the first involution
law (R6) hold trivially in $\ff A9$.

For the verification of the associative law (R4), observe that
most instances of this axiom yield the same result in $\ff A9$ as
in $\comz$.  Since $\comz$ is a relation algebra, these instances
must hold in $\f C$ and therefore also in    $\ff A9$. This
includes the following cases. (i) At least one of the elements
$r$, $s$, and $t$ is empty; in this case, both sides of (R4)
reduce to the empty set. (ii) All three elements are singletons of
group elements, say \[r=\{f\},\qquad s=\{g\},\qquad
t=\{h\}\semicolon\] in this case, both sides of (R4) reduce to the
singleton  $\{f\scir g\scir h\}$. (iii) At least two of the three
elements have cardinality at least two, and the third is not
empty; in this case, both sides of (R4) reduce to the unit
$\{0,1,2\}$. (iv) One of the elements is the unit, and the other
two are non-empty; in this case,   both sides of (R4) again reduce
to the unit. (v) Both $s$ and $t$ are singletons; in this case,
all relative products involved are  computed  the same way in $\ff
A9$ as   in $\f C$.

There remain two cases to consider. If $r$ and $s$ are singletons,
and $t$ a doubleton, then
\begin{multline*}
 r;(s;t)=r\mo\scir(s\mo\scir
t)=(r\mo\scir s\mo)\scir t\\=(s\scir r)\mo\scir t=(r\scir
s)\mo\scir t=(r;s);t\per
\end{multline*} The first and last equalities follow  from the definition of
$\,;\,$ and the assumption that $r$ and $s$ are singletons, and
$t$ a doubleton.  Notice in this connection that the relative
product of two singletons is always a singleton, and the relative
product of a singleton with a doubleton is always a doubleton,
both in $\ff A9$ and in $\f C$. The second and third equalities
follow from the validity of (R4) and (R7) in $\f C$. The fourth
equality uses the fact that the operation of relative
multiplication in $\f C$ is commutative (because the group
underlying~$\f C$ is commutative). If $r$ and $t$ are singletons,
and $s$ a doubleton, then
\[r;(s;t)=r\mo\scir(s\scir t)=(r\mo\scir s)\scir t=(r;s);t\comma\]
by the definition of relative multiplication in $\ff A9$, the
assumptions on the three elements, and the validity of (R4) in $\f
C$.

Turn now to the task of verifying the second involution law (R7)
in $\ff A9$. As in the case of (R4), most instances of (R7) yield
the same result in $\ff A9$ as in $\f C$, and are therefore
automatically valid in $\ff A9$. This includes the case  when at
least one of the elements $r$ and $s$ is empty,  in which case
both sides of (R7) reduce to the empty set; the case when one of
the elements is non-empty and the other is the unit $\{0,1,2\}$,
in which case both sides of (R7) reduce to the unit;  and the case
when both $r$ and~$s$ have at least two elements, in which case
both sides of (R7) again reduce to the unit. There remain three
cases to consider. If $r$ and $s$ are both singletons, then
\[(r;s)\ssm=r;s=s;r=s\ssm;r\ssm\per\] The first and last equalities use the fact that $r;s$ is  a singleton, and    converse
 is the identity function on singletons in $\ff A9$.  The
second equality follows from the fact that relative multiplication
in $\ff A9$ is commutative on singletons. If $r$ is a singleton,
and $s$ a doubleton, then
\[(r;s)\ssm=(r\mo\scir s)\ssm=(r\mo\scir s)\mo=s\mo\scir  (r\mo)\mo=s\mo\scir r=s\ssm;r=s\ssm;r\ssm\per\]
The first   equality follows from the definition  of relative
multiplication in $\ff A9$, and the assumption that $r$ is a
singleton and $s$ a doubleton.  The second equality uses
 the fact that the operations of
converse in $\f C$ and $\ff A9$ coincide on doubletons,
and~$r\mo\scir s$ must be a doubleton (since  $r$ is a singleton
and $s$ a doubleton). The third and fourth equalities use the
validity of~(R7) and (R6) in $\f C$. The fifth equality uses the
fact that the operations of converse in $\f C$ and $\ff A9$
coincide on doubletons, and so do the operations of relative
multiplication when the right-hand argument is a singleton.  The
sixth equality uses the fact that converse on singletons is the
identity function in $\ff A9$, and $r$ is assumed to be a
singleton. Finally, if $r$ is a doubleton, and $s$ a singleton,
then
\[(r;s)\ssm=(r\scir s)\mo=s\mo\scir r\mo=s; r\ssm=s\ssm;r\ssm\per\]
The first and third equalities use the definitions of relative
multiplication and converse in $\ff A9$, and the assumptions on
$r$ and $s$; the second equality uses the validity of~(R7) in $\f
C$; and the last equality uses the fact that converse is the
identity function on singletons in $\ff A9$.

Next, we verify the distributive law (R8) in $\ff A9$.  As usual,
most
 instances of this axiom yield the same result in $\ff A9$ as
in $\f C$, and are therefore valid in $\ff A9$.  This includes all
cases when $t$ is not a doubleton.  It also includes  the case
when $t$ is a doubleton and at least one of $r$ and $s$ has at
least two elements (in which case, both sides of (R8) reduce to
the unit, because in the relative product in $\ff A9$ and in $\f
C$  of two elements with at least two elements is always the
unit---see \refTa{cz3}). The case when $t$ is a doubleton, and at
least one of $r$ and $s$ is the empty set is trivial; for example,
if $r$ is the empty set, then
\[(r+s);t=s;t\qquad\text{and}\qquad r;t+s;t=\varnot;t+s;t=\varnot +s;t=s;t\per\]
Similarly, the case when $r=s$ is trivial.  There remains the case
when $t$ is a doubleton, and $r$ and $s$ are distinct singletons.
In this case, $r+s$ is a doubleton, so $(r+s);t$ is the unit. As
$r;t$ and $s;t$ are distinct doubletons, the sum $r;t+s;t$ is also
the unit. Thus, (R8) is valid in $\ff A9$ in this case as well.

It remains to verify Tarski's law (R10), or equivalently,
\refEq{r10}, in $\ff A9$. The instances of \refEq{r10} in which
$r$ is the empty set or has at least two elements yield the same
result in $\ff A9$ as in $\f C$, and are therefore valid in $\ff
A9$. The same is true of those instances of \refEq{r10} in which
$r$ is a singleton, and $s$ is either the empty set or the unit.
There remain two cases to consider. If $r$ and $s$ are both
singletons, then~$-(r;s)$ in $\ff A9$ coincides with $-(r\scir s)$
in $\f C$ and  is therefore a doubleton.  It follows that
\[r\ssm;-(r;s)=r;-(r;s)=r\mo\scir-(r\scir s)\le -s\per\] The first
equality uses the fact that converse is the identity operation on
singletons in $\ff A9$, the second uses the definition  of
relative multiplication in $\ff A9$ and the assumptions on $r$ and
$s$, and the last uses the validity of \refEq{r10} in $\f C$.  If
$r$ is a singleton, and $s$ a doubleton, then $r;s$ is also a
doubleton, so $-(r;s)$ is a singleton.  Consequently,
\[r\ssm;-(r;s)=r;-(r;s)=r\scir -(r\mo\scir s)=(r\mo)\mo \scir -(r\mo\scir s)\le-s\per\] The
first equality uses the fact that converse is the identity
function on singletons in $\ff A9$,  while the second equality
uses the definition of relative multiplication in $\ff A9$, the
assumptions on $r$ and $s$,  and the observations preceding the
calculation. The third equality uses the validity of (R6) in $\f
C$, and the final inequality follows from the validity of
\refEq{r10} in $\f C$ (with $r$ replaced by $r\mo$).

A computational check using Mace4 has shown that $\ff A9$ is the
unique independence model for (R9) of cardinality $8$, and that
there is no  smaller independence model for this axiom.

It is interesting to note that the left-hand distributive law
\refEq{r8} for relative multiplication  fails in $\ff A9$. Indeed,
take $r$, $s$, and $t$ to be $\{1\}$, $\{0\}$, and~$\{2\}$
respectively to obtain
\begin{gather*}
  r;(s+t)=\{1\};(\{0\}+\{2\})=\{1\};\{0,2\}=\{1,2\}\comma\\ \intertext{by the definition of $\,;\,$ (see \refTa{mmr9}), but}
  r;s+r;t=\{1\};\{0\}+\{1\};\{2\}=\{1\}\scir\{0\}+\{1\}\scir\{2\}=\{1\}+\{0\}=\{0,1\}\per
\end{gather*}

\section{Independence of (R10)}
Consider any Boolean   algebra $(A\smbcomma +\smbcomma -)$ with at
least two elements.    Take relative multiplication to be the
Boolean operation of addition, take converse to be the identity
function on $A$, and take the identity element to be the Boolean
zero. To see that (R10) fails in the resulting algebra $\ff
A{10}$, take  $r$ and $s$ to be $1$, and observe that
\[r\ssm;-(r;s)=r+-(r+s)=1+-(1+1)=1\not\le 0=-s\comma\]
by the definitions of relative multiplication and converse, and
the choice of $r$ and~$s$. On the other hand, the Boolean axioms
(R1)--(R3) hold automatically in $\ff A{10}$, and (R4)--(R9)
reduce to Boolean laws, so they, too, are valid in $\ff A{10}$. To
give two concrete examples, consider (R7) and (R8). We have
\begin{gather*}
  (r;s)\ssm=r+s=s+r=s\ssm;r\ssm\comma\\
  \intertext{and}(r+s);t=(r+s)+t= (r+t)+(s+t)=r;t + s;t\comma
\end{gather*}
for all elements $r$, $s$, and $t$, by the definition of relative
multiplication and converse.

Start with a two-element Boolean algebra in the preceding
construction to arrive at an independence model for (R10) of
minimal cardinality.

There is  another interesting  and rather different independence
model for (R10) that is worthwhile discussing. Start with the set
relation algebra $\ff M3$ (see \refTa{minsquare}), and modify the
operation of relative multiplication in two ways: require relative
multiplication by the diversity element to always yield the
diversity element,
\begin{gather*}
  r;\diver=\diver;r=\diver\\
  \intertext{for all elements $r$, and require}
  0;1=1;0=\diver
\end{gather*}
(see \refTa{b10}).

\begin{table}[htb]
\begin{center} \begin{tabular}{ l | c | c | c | c |} $;$ &
\,$0$\, & $\ident$ & $\diver$ & \,$1$\,\\ \hline $0 $ & $0$ & $0$
& $\diver$ & $\diver$\\ \hline $\ident$ & $0$ & $\ident$ &
$\diver$ & $1$\\ \hline $\diver$ & $\diver$ & $\diver$ & $\diver $
& $\diver $\\ \hline $1$\, & $\diver$ & \,$1$\, & $\diver $ &
$1$\\ \hline
\end{tabular}
\end{center} \vspace{.1in}\caption{Relative multiplication table for $\ff B{10}$\per}\label{Ta:b10}
\end{table}

 To see that (R10), or equivalently, \refEq{r10}, fails in the resulting algebra  $\ff B{10}$, take~$r$ and $s$ to be
 $\diver$ to  obtain
\begin{equation*}
r\ssm;-(r;s)=\diver\ssm;-(\diver;\diver)=\diver;-\diver
=\diver;\ident=\diver\not\le \ident=-s\comma
\end{equation*} by the definition of relative multiplication and
converse, and the choice of $r$ and $s$.

The Boolean part of $\ff B{10}$ coincides with the Boolean part of
$\ff M3$, the operation of relative multiplication is commutative,
$\ident$ remains the identity element for relative multiplication,
and
 converse is the identity function, so
(R1)--(R3),~(R5),~(R6),~(R7), and (R9) are all easily seen to be
valid in $\ff B{10}$. It remains to verify (R4) and~(R8). The
relative multiplication table for $\ff B{10}$ differs from that of
$\ff M3$ (compare Tables~\ref{Ta:minsquare} and \ref{Ta:b10})
 in the seven entries
\begin{equation*}%\tag{2}\label{Eq:1-00008.2}
0;\diver,\quad \diver;0,\quad 0;1,\quad
1;0,\quad\diver;\diver,\quad\diver;1,\quad\text{and} \quad
1;\diver\comma
\end{equation*} which all have the value $\diver$ in $\ff B{10}$.
Consequently, every instance of (R4) and~(R8) which does not
involve the computation of one of these products  is automatically
valid in $\ff B{10}$, because it is valid in $\ff M3$.  Notice
also that the relative product of two elements in $\ff B{10}$ is
never $\diver$ unless one of the elements is $\diver$, or else one
of the elements is $0$ and the other is $1$ (see \refTa{b10}).
Consequently, if $r$ and $s$ are both different from~$\ident$,
then $r;s$ is different from $\diver$ if and only if $r$ and $s$
are either both~$0$ or both $1$.

The validity in $\ff B{10}$ of the associative law (R4)  follows
readily from the preceding observations. Any instance of (R4) in
which at least one of the three elements~$r$,~$s$, and $t$ is
$\ident$  holds trivially in $\ff B{10}$, because both sides of
(R4) reduce to the relative product of the other two elements.
Assume now that none of these three elements is~$\ident$.  In this
case, none of the relative products involved in (R4) can have the
value~$\ident$ (see \refTa{b10}), so the left side of (R4) is
different from $\diver$ if and only if~$r$ and $s;t$ are either
both $0$ or both $1$, by the observations at the end of the
preceding paragraph.   In the case under consideration, $s;t$ can
only be $0$ or $1$ if $s$ and $t$ are both $0$ or both $1$
respectively (see \refTa{b10}). Consequently, the left side of
(R4) is different from $\diver$ if and only if $r$, $s$, and $t$
are all $0$ or all $1$.  A similar remark applies to the right
side of (R4). Thus, either both sides of (R4) evaluate to
$\diver$, in which case (R4) holds in $\ff B{10}$, or else $r$,
$s$, and $t$ all have the same value---either $0$ or $1$---and in
this case (R4) holds in $\ff B{10}$, because it holds in $\ff M3$.

Turn finally  to  the verification of (R8) in $\ff B{10}$. It is
to be shown that both sides of this axiom evaluate to the same
element in $\ff B{10}$. If $t$ is~$\diver$, then both sides
evaluate to $\diver$, and if $t$ is $\ident$, then both sides
evaluate to $r+s$. Consider next the case when~$t$ is $0$. If at
least one of $r$ and $s$ is $\diver$ or $1$, then both sides
of~(R8) evaluate to $\diver$. For example, if $r$ is $\diver$,
then
\[(r+s);t=(\diver +s);0=\diver\quad\text{and}\quad r;t+s;t=\diver;0+s;0=\diver\comma\]
since $s;0$ is at any rate below $\diver$.  A similar argument
applies when~$r$ is~$1$. The only other possibility in the case
under consideration is that~$r$ and $s$ both assume values in the
set $\{0,\ident\}$, and in this case the computation of each side
of (R8) yields the same result in $\ff B{10}$ as it does in $\ff
M3$.

There remains the case when $t$ is $1$. Keep in mind that~$0;1$
and~$\diver;1$ are both~$\diver$, and~$\ident;1$ and~$1;1$ are
both $1$ (see \refTa{b10}). The sum~$r+s$ assumes one of  four
values: $0$, $\ident$, $\diver$, or $1$. If this value is
$\ident$, then at least one of $r$ and $s$ must be~$\ident$
(since~$\ident$ is an atom), so both sides of (R8) evaluate to
$1$. If the value of the sum is~$1$, then at least one of $r$ and
$s$ is either $\ident$ or $1$, so both sides of (R8) again
evaluate to $1$. If the value of the sum is  $0$ or $\diver$, then
neither $r$ nor $s$ can be $\ident$ or $1$, and therefore both
sides of~(R8) must
 evaluate to $\diver$.
 This completes the
verification of~(R8) in $\ff B{10}$.

\section{A variant of Tarski's axiom system}\label{S:variant}

  Somewhat surprisingly, it turns out that  by modifying slightly one of the  axioms in Tarski's  system, namely (R8), another of the
   axioms, namely (R7), becomes redundant.  We begin
with some lemmas that will  be needed again later.
\begin{lemma}\label{L:lm1}  Under the assumption of \textnormal{(R1)--(R3)}\comma  axiom \textnormal{(R8)} implies the left-hand monotony law
for relative multiplication\comma and  axiom
\textnormal{\refEq{r8}} implies the right-hand hand monotony law
for relative multiplication\per
\end{lemma}
\begin{proof}
 If  $r\le s$, then $s=r+s$, by the definition of $\,\le\,$, and
 therefore
 \[s;t=(r+s);t=r;t+s;t\comma\] by (R8).  Consequently, $r;t\le
 s;t$, by the definition of $\,\le\,$.  This proves that (R8) implies the
 left-hand monotony law for relative multiplication.  A similar
 argument shows
 that \refEq{r8} implies the right-hand monotony law for relative
 multiplication.\end{proof}
\begin{lemma}\label{L:lm2}
   Under the assumption of \textnormal{(R1)--(R3)}\comma \textnormal{(R6)}, and
   \textnormal{\refEq{r8}}\comma axiom \textnormal{(R10)} is equivalent to  the
   law
\begin{equation*}\tag{R11$'$}\label{Eq:r12}
(r;s)\cdot t=0\qquad \text{if and only if}\qquad (r\ssm;t)\cdot
 s=0\per
\end{equation*}
\end{lemma}
\begin{proof}As was mentioned in \refS{asec1.1}, on the basis of (R1)--(R3), axiom (R10) is equivalent to
\refEq{r10}, so it suffices to prove that \refEq{r10} is
equivalent to \refEq{r12}.  Assume first that \refEq{r10} holds.
  If $(r;s)\cdot t=0$, then
    $t\le -(r;s)$, by Boolean algebra (here we are using (R1)--(R3)), and
therefore
\[r\ssm;t\le r\ssm;-(r;s)\le -s\comma\] by  the right-hand
monotony law for relative multiplication (which is valid under the
assumption of \refEq{r8}, by  \refL{lm1}) and~(R10). Consequently,
$(r\ssm;t)\cdot s=0$, by Boolean algebra. This argument
establishes the implication from left to right in~\refEq{r12}.

To establish the reverse implication, assume~$(r\ssm;t)\cdot
 s=0$\comma and use the results of the previous paragraph (with $r\ssm$, $t$, and $s$ in place
 of $r$, $s$, and $t$ respectively) to obtain $(r\ssm\conv;s)\cdot t=0$.  Apply (R6) to conclude that $(r;s)\cdot
 t=0$.

 Assume now that \refEq{r12}  holds.  Take $t$ to be $-(r;s)$ and observe that the left side of \refEq{r12} obviously holds,
 by Boolean algebra. Consequently, the right side must hold, that is to say,
 \[[r\ssm;-(r;s)]\cdot s=0\per\] This equation is clearly equivalent to
 \refEq{r10}, by Boolean algebra.\end{proof}

As is clear from the proof of \refL{lm2}, under the assumption of
(R1)--(R3) and~\refEq{r8}, axiom (R10) is equivalent to the
implication from left to right in \refEq{r12}, that is to say,
(R10) is equivalent to \refEq{r11} (see \refS{asec1.1}).

\begin{lemma}\label{L:lm3.1}  Under the assumption of \textnormal{(R6), (R7)} and
\textnormal{(R9)}\comma  axiom \textnormal{(R8)} is equivalent to \textnormal{\refEq{r8}}\per
\end{lemma}
\begin{proof}  The derivation of \refEq{r8} from (R8) is contained in the proof of Theorem 1.21 in
Chin-Tarski\,\cite{ct}. For the convenience of
the reader, here are the details of the argument.
 Observe that
\begin{equation*}\tag{1}\label{Eq:thm1.2}
[(r\ssm+s\ssm);t\ssm]\ssm=t\ssm\conv;(r\ssm+s\ssm)\ssm=t\ssm\conv;(r\ssm\conv+s\ssm\conv)=t;(r+s)\comma
\end{equation*}
by (R7), (R9), and (R6), and
\begin{multline*}\tag{2}\label{Eq:thm1.3}
[(r\ssm;t\ssm)+(s\ssm;t\ssm)]\ssm=(r\ssm;t\ssm)\ssm+(s\ssm;t\ssm)\ssm\\=t\ssm\conv;r\ssm\conv+t\ssm\conv
s\ssm\conv=t;r+t;s\comma
\end{multline*} by (R9), (R7), and (R6). Axiom (R8)   (with~$r$,~$s$, and~$t$ replaced by $r\ssm$,
$s\ssm$, and $t\ssm$ respectively) ensures that
\begin{equation*}%\tag{3}\label{Eq:thm1.1}
(r\ssm+s\ssm);t\ssm=r\ssm;t\ssm +s\ssm;t\ssm\per
\end{equation*} Form the converse of both sides of this last equation,
and use \refEq{thm1.2} and \refEq{thm1.3} to arrive at \refEq{r8}.

A dual argument leads to an analogous derivation of (R8) from \refEq{r8}.
\end{proof}

The next lemma occurs as part of Theorem 313 in \cite{ma06}.  We prove it here for the convenience of the reader.

\begin{lemma}\label{L:lm3.2}  Under the assumption of \textnormal{(R1)--(R3)}\comma
axioms  \textnormal{(R4)} and \textnormal{(R5)}\comma together
with \textnormal{\refEq{r12}}\comma
 imply \textnormal{(R7)}\per
\end{lemma}
\begin{proof}The key step in
the argument is the proof of the  equivalence
\begin{align*}
  (r;s)\ssm\cdot t=0\qquad&\text{if and only if}\qquad
  (s\ssm;r\ssm)\cdot t=0\per\tag{1}\label{Eq:lm9.2}\\
  \intertext{for all elements $r$, $s$, and $t$.  To establish \refEq{lm9.2}, observe that}
   (r;s)\ssm\cdot t=0\qquad&\text{if and only if}\qquad [(r;s)\ssm;\ident]\cdot t=0\comma\\
   &\text{if and only if}\qquad [(r;s);t] \cdot
  \ident=0\comma\\
  &\text{if and only if}\qquad [r;(s;t)] \cdot
  \ident=0\comma\\
  &\text{if and only if}\qquad (r\ssm;\ident)\cdot(s;t)=0\comma\\
 &\text{if and only if}\qquad  r\ssm \cdot(s;t)=0\comma\\
 &\text{if and only if}\qquad  (s;t)\cdot r\ssm=0\comma\\
 &\text{if and only if}\qquad  (s\ssm; r\ssm)\cdot t=0\per
\end{align*}
\noindent The first equivalence uses (R5), the second uses~\refEq{r12} (with $r;s$, $t$, and $\ident$ in place of $r$, $s$, and $t$ respectively), the third uses~(R4),  the fourth
uses~\refEq{r12} (with~$s;t$ and $\ident$  in place of $s$ and $t$ respectively), the fifth uses (R5),
the sixth uses Boolean algebra, and the seventh uses~\refEq{r12}
(with $s$, $t$, and $r\ssm$ in place of $r$, $s$ and~$t$   respectively).

Turn now to the proof of the second involution law. Obviously,
\begin{align*}
(r;s)\ssm\cdot-[(r;s)\ssm]&=0\comma\\ \intertext{by Boolean
algebra, so} (s\ssm;r\ssm)\cdot -[(r;s)\ssm]&=0\comma
\end{align*}  by \refEq{lm9.2} (with
$-[(r;s)\ssm]$ in place of $t$).  It follows by Boolean algebra
that
\begin{equation*}\tag{2}\label{Eq:lm9.3}
s\ssm;r\ssm\le(r;s)\ssm\per
\end{equation*}
 Similarly, it is obvious that
\begin{align*}
(s\ssm;r\ssm)\cdot-(s\ssm;r\ssm)&=0\comma\\ \intertext{by Boolean
algebra.\, so} (r;s)\ssm\cdot -(s\ssm;r\ssm)&=0\comma
\end{align*}
by \refEq{lm9.2} (with $-(s\ssm;r\ssm)$ in place of $t$). It
follows by Boolean algebra that
\begin{equation*}\tag{3}\label{Eq:lm9.4}
(r;s)\ssm\le s\ssm;r\ssm\per
\end{equation*}
 Combine \refEq{lm9.3} and \refEq{lm9.4} to arrive at the second
involution law.
\end{proof}

Take $\mc R$ to be the system of equations obtained from
(R1)--(R10) by dropping~(R7), and replacing the right-hand
distributive law  (R8) with its left-hand version~\refEq{r8}.

\begin{theorem}\label{T:tv}
  The system of axioms $\mc R$ is equivalent to Tarski's system
  \textnormal{(R1)--(R10).}
\end{theorem}
\begin{proof} It is easy to check that Tarski's  axioms imply  the axioms
in $\mc R$.  In fact, it is only necessary to derive \refEq{r8}
from (R1)--(R10), and this is done in \refL{lm3.1}.

To prove that, conversely, the axioms in $\mc R$ imply Tarski's
axioms, it must be shown that (R7) and (R8) are derivable from
$\mc R$. Apply \refL{lm2} to obtain
\refEq{r12}, and apply \refL{lm3.2} to obtain (R7).  An
application of \refL{lm3.1} now yields (R8).
\end{proof}

\section{The independence of axiom system $\mc R$}\label{S:indepr}

Interestingly, the axioms in $\mc R$ are also all independent of
one another.  For example, the left-hand distributive law
\refEq{r8} fails in the  model $\ff A8$,  while the remaining
axioms of~$\mc R$ are valid in $\ff A8$, so \refEq{r8} is
independent of the other axioms of $\mc R$. In fact, the same
assignment of values to $r$, $s$, and $t$ that invalidates (R8) in
$\ff A8$ also invalidates \refEq{r8}, since relative
multiplication and addition are commutative operations in $\ff
A8$. Alternatively,~\refEq{r8} must fail in the independence model
$\ff A7$, in which the remaining axioms of $\mc R$ are valid.
Indeed, if \refEq{r8} were valid in $\ff A7$, then~$\ff A7$ would
be a model of $\mc R$, and therefore also of (R7), by \refT{tv};
but we have seen that this is not the case. To obtain a concrete
instance in which (R8) fails, let $r$ be any non-zero element in
$\ff A7$, and let $s$ and $t$ be $\ident$ and $\diver$
respectively.  The definition of relative multiplication in $\ff
A7$   implies that
\begin{align*}
r;(s+t)&=r;(\ident+\diver)=r;1=0\\ \intertext{but}
r;s+r;t&=r;\ident+r;\diver=r+0=r\neq 0\per
\end{align*}    This argument actually shows more than is claimed.
Since (R8) is valid in $\ff A7$,  axiom \refEq{r8} is independent
of the set of axioms (R1)--(R6), (R8), (R9), and (R10).  We will
need
  this observation   later.

 As regards the
independence of  (R$n$) in $\mc R$ for $1\le n\le 6$ and $n= 10$,
the left-hand distributive law~\refEq{r8} is valid in the
independence model $\ff An$ constructed above, so~$\ff An$ also
serves to establish the independence of (R$n$) with respect
to~$\mc R$. However, the left-hand distributive law fails in $\ff
A9$, so a new model must be constructed in order to establish the
independence of~(R9) with respect to $\mc R$.

 The
independence model $\ff B9$ for (R9) with respect to $\mc R$ is
obtained from the relation algebra $\f D$ constructed in
\refTa{b4} by modifying the definitions of relative multiplication
and converse. In order to avoid confusion of notation,  write
$\,;\,$ and~$\,\conv\,$ for the operations of relative
multiplication and converse to be defined in $\ff B9$, and write
$\,\scir$ for the operation of relative multiplication in $\f D$;
a separate notation for the operation of converse in $\f D$ is
unnecessary, since this operation is defined to be the identity
function. In $\ff B9$, converse is defined to interchange the
elements
\[\ident +a\qquad\text{and}\qquad\ident +b\comma\] and to map
every other element  to itself. (Notice the similarity in
intuition  with the model $\ff A9$.)  As a result, (R9) fails in
$\ff B9$. Indeed, if $r$ and $s$ are taken to be $\ident$ and $a$
respectively, then
\[(r+s)\ssm=(\ident +a)\ssm=\ident+b\neq \ident
+a=r\ssm+s\ssm\per\]

As in the case of the algebra $\ff A9$, the change in the
definition of converse requires a corresponding compensatory
change in the definition of relative multiplication.  If $r$ is
one of the elements $\ident +a$ and $\ident +b$, and $s$ is one of
the atoms $a$ and $b$, then the relative product $r;s$ in $\ff B9$
is defined by
\[r;s=r\ssm\scir s\per\]  In all other cases, relative
multiplication in $\ff B9$ is defined to coincide with relative
multiplication in $\f D$.  Thus, only  four entries in \refTa{b4}
are changed in the passage from $\f D$ to $\ff B9$ (see
\refTa{b5}). Notice that the preceding equation is actually valid
for all choices  $r$ and $s$ except   when $r$ is one of $\ident
+a$ and $\ident +b$, and $s$ is $\ident$. Indeed, if $r$ is
different from $\ident +a$ and $\ident +b$, then $r;s$ and $r\ssm$
coincide  with $r\scir s$ and $r$ respectively, by definition,  so
that\[r;s=r\scir s= r\ssm\scir s\per\]  On the other hand, if $r$
is one of  $\ident +a$ and $\ident +b$, and $s$ is different from
$\ident$, then either $s$ is $0$, in which case the desired
equality holds trivially; or $s$ is one of $a$ and~$b$, in which
case the   equality holds by definition; or $s$ is the sum of at
least two atoms, in which case both $r;s$ and $r\ssm\scir s$ are
equal to $1$, yielding again the desired equality (see
\refTa{b4}).
\begin{table}[htb]
  \centering\begin{tabular}{|c|c|c|}\hline
    % after \\: \hline or \cline{col1-col2} \cline{col3-col4} ...
     $\,\scir\,$& $a$ & $b$ \\ \hline
        $\ident+a$ & $1$ & $\diver$
  \\ \hline $\ident+b$ &
  $\diver$ & $1$     \\ \hline
 \end{tabular}\qquad\qquad
  \begin{tabular}{|c|c|c|}\hline
    % after \\: \hline or \cline{col1-col2} \cline{col3-col4} ...
     $\,;\,$& $a$ & $b$ \\ \hline
        $\ident +a$ & $\diver$ & $1$
  \\ \hline $\ident +b$ &
  $1$ & $\diver$     \\ \hline
 \end{tabular}\vspace{.1in}
  \caption{Comparison of the  differences in the relative multiplication tables   for   $\f D$ and $\ff B9$.}\label{Ta:b5}
\end{table}

Axioms (R1)--(R3),  (R5), and (R6)  obviously all hold in $\ff
B9$, so it remains to verify the validity of (R4), (R7),
\refEq{r8}, and (R10). Consider first (R4).  If one of the
elements $r$, $s$, and $t$ is $0$, then both sides of (R4) reduce
to $0$ in $\ff B9$, and if one of these elements is $\ident$, then
both sides of (R4) reduce to the relative product of the other two
elements.  In all other cases, both sides of (R4) reduce to $1$ in
$\ff B9$.  In more detail, the relative product of two elements
different from $0$ and $\ident$ is either $\diver$ or $1$, and the
relative product of these last two elements with any  element
different from $0$ and $\ident$ is always $1$ (see Tables
\ref{Ta:b4} and \ref{Ta:b5}).

Turn now to the verification of (R7). If $r$ or $s$ is $\ident$,
then both sides of (R7) reduce to $s\ssm$ or $r\ssm$ respectively,
so in this case (R7) holds trivially.  In all other cases, we have
\[(r;s)\ssm=(r\ssm\scir s)\ssm=r\ssm\scir s=s\scir r\ssm=s\ssm\conv\scir r\ssm=s\ssm;r\ssm\per\]
The first and last equalities follow  from the definition of
relative multiplication in~$\ff B9$ and the assumption that
neither $r$ nor $s$ is $\ident$ (see the observations made above).
The second equality uses the fact that $r\ssm\scir s$ is either
$0$, $\diver$, or $1$ in all cases under consideration, and
converse maps each of these elements to itself in $\ff B9$. The
third equality uses the fact that relative multiplication in $\f
D$ is commutative, and the fourth  uses the validity of (R6) in
$\ff B9$.

The next task is the verification of \refEq{r8}.  For all values
of $r$ except $\ident+a$ and~$\ident + b$, the computation of both
sides of \refEq{r8} is the same in $\ff B9$ as it is in $\f D$, so
these instances of \refEq{r8} are all valid in $\ff B9$. Also, if
$t$ is $0$, or if $s=t$, then both sides of~\refEq{r8} reduce to
$r;s$, and analogously if $s$ is $0$, so   these instances of
\refEq{r8} are also valid in $\ff B9$. Assume now that $r$ is one
of
  $\ident + a$ and $\ident + b$, and that $s$ and~$t$ are distinct non-zero
  elements. If neither $s$ nor $t$ is $\ident$, then
\begin{equation*}%\tag{1}\label{Eq:b9.0}
 r;(s+t)=r\ssm\scir(s+t)=r\ssm\scir s +r\ssm\scir t=r;s+r;t\comma
\end{equation*}
  by the definition of relative multiplication
  in $\ff B9$ and the validity of \refEq{r8} in $\f D$, so these instances of \refEq{r8} hold in $\ff B9$.
   On the other hand, if $s$ is $\ident$, then $r; s$ is $r$, which is  above $\ident$ by assumption; and   $t$ is
  a non-zero element different from $\ident$, by assumption; so $r;t$ must either
  be $\diver$ or $1$, and therefore \[r;s+r;t\ge \ident+\diver=1\per\]  Since  $s+t$ is the sum of at least two atoms,
  \[r;(s+t)=1\comma\] by \refTa{b4}, and therefore all such instances of \refEq{r8} hold in $\ff B9$ as well.  A similar argument applies if $t$ is
  $\ident$. This completes the verification of \refEq{r8} in $\ff
  B9$.

Turn finally to the verification of (R10), or equivalently,
\refEq{r10}. If $r$ is different from both $\ident +a$ and $\ident
+b$, then the computation of
\begin{equation*}\tag{1}\label{Eq:b9.1}
r\ssm;-(r;s)
\end{equation*}
  is the same in $\ff B9$ as it is in $\f D$, and consequently
 \refEq{b9.1}
must be below $-s$, by the validity of \refEq{r10} in $\f D$.
Suppose now that $r$ is one of $\ident+a$ and $\ident +b$.  If~$s$
is~$0$, then~$-s$ is $1$, so obviously \refEq{b9.1} is below $-s$.
If $s$ is the sum of at least two atoms, then~$r;s$ is $1$, and
therefore $-(r;s)$ is $0$ (see \refTa{b4}).
Consequently,~\refEq{b9.1} reduces to~$0$, which is below $-s$.
If $s$ is $\ident$, then \refEq{b9.1} reduces to $r\ssm;-r$, which
in the cases under consideration must yield $\diver$, by
\refTa{b5} and the definition of converse in $\ff B9$ (since in
this case $-r$ is the subdiversity atom that is below $r\ssm$).
Also, $-s$ is~$\diver$, so \refEq{b9.1} is equal to $-s$.  In the
remaining cases, $s$ is one of $a$ and $b$.  Consequently,~$r;s$
assumes one of two values $\diver$ or $1$, according to whether
$s$ is, or is not, the subdiversity atom below $r$ (see
\refTa{b5}).  In the first case, $-(r;s)$ is~$\ident$, so
\refEq{b9.1} reduces to $r\ssm$, which coincides with $-s$,  by
\refTa{b5} and the definition of relative multiplication. For
example, if $r$ is $\ident+a$ and $s$ is $a$, then
\begin{multline*}
r\ssm;-(r;s)=(\ident +a)\ssm;-((\ident+a);a)=(\ident +a)\ssm;-
\diver\\=(\ident +a)\ssm;\ident=(\ident +a)\ssm=(\ident +b)=-s\per
\end{multline*}
In the second case, $r;s$ is $1$, by \refTa{b5}, so  $-(r;s)$ is
$0$, and therefore \refEq{b9.1} reduces to $0$, which is certainly
below $-s$.  This completes the verification of \refEq{r10}.

The following theorem has been proved.
\begin{theorem}\label{T:indepr}
The set of axioms $\mc R$ is independent\per
\end{theorem}

Notice that (R8) fails in $\ff B9$.  For instance, if $r$, $s$,
and $t$ are $\ident$, $a$, and $b$ respectively\comma then
\begin{equation*}
  (r+s);t=(\ident +a);b=1\qquad\text{and}\qquad
  r;t+s;t=\ident;b+a;b=b+\diver=\diver\per
\end{equation*}

A computational check using Mace4 has shown that there is no
independence model for (R9) of cardinality less than $8$, so $\ff
B9$ is a minimal  independence model for this axiom with respect
to the axiom system $\mc R$.

\section{A second variant of Tarski's axiom system}\label{S:sys.
s}

As was   pointed out above, in  the   independence models $\ff A9$
and $\ff B9$ for (R9), the right-hand and left-hand distributive
laws for relative multiplication fail respectively. This raises
the  question of whether (R9) is derivable
from~(R1)--(R8),~\refEq{r8}, and (R10). As it turns out, (R9) is
so derivable, and in fact even more is true: if~\refEq{r8} is
added to Tarski's original axiom system, then both (R7) and (R9)
become redundant in the sense that they are both derivable from
the remaining axioms of the system.

\begin{lemma}\label{L:lm10}
 Under the assumption of \textnormal{(R1)--(R3)}\comma
axioms  \textnormal{(R5)} and \textnormal{(R8)}\comma together
with \textnormal{\refEq{r12}}\comma
 imply \textnormal{(R9)}\per
\end{lemma}
\begin{proof}The proof is very similar  to the proof of \refL{lm3.2}.  The key step in
the argument is the proof of the  equivalence
\begin{align*}
  (r+s)\ssm\cdot t=0\qquad&\text{if and only if}\qquad
  (r\ssm+s\ssm)\cdot t=0\per\tag{1}\label{Eq:lm10.2}\\
  \intertext{for all elements $r$, $s$, and $t$.  To establish \refEq{lm10.2}, observe that}
   (r+s)\ssm\cdot t=0\qquad&\text{if and only if}\qquad [(r+s)\ssm;\ident]\cdot t=0\comma\\
   &\text{if and only if}\qquad [(r+s);t] \cdot
  \ident=0\comma\\
  &\text{if and only if}\qquad [(r;t)+(s;t)] \cdot
  \ident=0\comma\\
  &\text{if and only if}\qquad  (r;t)\cdot\ident +(s;t)\cdot\ident=0\comma\\
  &\text{if and only if}\qquad (r\ssm;\ident)\cdot t+ (s\ssm ;\ident)\cdot t=0\comma\\
 &\text{if and only if}\qquad  r\ssm \cdot t +  s\ssm\cdot t=0\comma\\
 &\text{if and only if}\qquad  (r\ssm+s\ssm)\cdot t=0\per
\end{align*}
\noindent The first and sixth equivalences use  (R5), the second
uses~\refEq{r12} (with $r+s$, $t$, and~$\ident$ in place of $r$,
$s$, and $t$ respectively), the third uses~(R8),  the fourth and
seventh use  Boolean algebra, and the fifth uses~\refEq{r12} twice
(the first time with $t$ and $\ident$ in place of $s$ and $t$
respectively, and the second time with $s$,~$t$ and $\ident$ in
place of~$r$,~$s$ and~$t$ respectively).

Turn now to the proof of the second involution law. Obviously,
\begin{align*}
(r+s)\ssm\cdot-[(r+s)\ssm]&=0\comma\\ \intertext{by Boolean
algebra, so} (r\ssm+s\ssm)\cdot -[(r+s)\ssm]&=0\comma
\end{align*}  by \refEq{lm10.2} (with
$-[(r+s)\ssm]$ in place of $t$).  It follows by Boolean algebra
that
\begin{equation*}\tag{2}\label{Eq:lm10.3}
r\ssm+s\ssm\le(r+s)\ssm\per
\end{equation*}
 Similarly, it is obvious that
\begin{align*}
(r\ssm+s\ssm)\cdot-(r\ssm+s\ssm)&=0\comma\\ \intertext{by Boolean
algebra, so} (r+s)\ssm\cdot -(r\ssm+s\ssm)&=0\comma
\end{align*}
by \refEq{lm10.2} (with $-(r\ssm+s\ssm)$ in place of $t$). It
follows by Boolean algebra that
\begin{equation*}\tag{3}\label{Eq:lm10.4}
(r+s)\ssm\le r\ssm+s\ssm\per
\end{equation*}
 Combine \refEq{lm10.3} and \refEq{lm10.4} to arrive at (R9).
\end{proof}

Take $\mc S$ to be the axiom system    consisting of equations
(R1)--(R6), (R8), \refEq{r8}, and (R10).
\begin{theorem}\label{T:tv2}
  The system of axioms $\mc S$ is equivalent to Tarski's system
  \textnormal{(R1)--(R10).}
\end{theorem}
\begin{proof} It is easy to check that Tarski's  axioms imply  the axioms
in $\mc S$.  In fact, it is only necessary to derive \refEq{r8}
from (R1)--(R10), and that is done in \refL{lm3.1}.

To prove that, conversely, the axioms in $\mc S$ imply Tarski's
axioms, it must be shown that (R7) and (R9) are derivable from
$\mc S$. Observe first that \refEq{r12} is derivable from $\mc S$,
by \refL{lm2}. Consequently,  (R7) is derivable from $\mc S$, by
\refL{lm3.2}, and  (R9) is derivable from $\mc S$, by \refL{lm10}.
\end{proof}

\section{The independence of axiom system $\mc S$}\label{S:indeps}

The axioms in $\mc S$ are   independent of one another. Indeed, as
was already pointed out in the first paragraph of \refS{indepr},
the left-hand distributive law \refEq{r8} fails in the model $\ff
A7$, while the remaining axioms of~$\mc S$ are valid in $\ff A7$.
Consequently,~\refEq{r8} is independent of the other axioms of
$\mc S$.  Similarly, it was proved in \refS{indepr} that~(R8)
fails in the model $\ff B9$, while (R1)--(R6), \refEq{r8}, and
(R10) all hold in that model. Consequently, (R8) is independent of
the remaining axioms in $\mc S$.
 As regards the
independence of axioms (R$n$) in $\mc S$ for $1\le n\le 6$ and $n=
10$, axiom~\refEq{r8} is valid in the independence model $\ff An$,
so $\ff An$ also serves to establish the independence of~(R$n$) in
$\mc S$.
\begin{theorem}\label{T:ins}
  The set of axioms $\mc S$ is independent\per
\end{theorem}

\end{document}